\documentclass[reqno]{amsart}

\usepackage{amsfonts}
\usepackage{amsmath}
\usepackage{amssymb}
\usepackage{latexsym}
\usepackage{amscd}
\usepackage{amsthm}
\usepackage{mathrsfs}
\usepackage{subfigure}
\usepackage{graphicx}
\usepackage[all]{xy}
\usepackage{color}

\newtheorem{thm}{Theorem}[section]
\newtheorem{prop}[thm]{Proposition}

\newtheorem{lem}[thm]{Lemma}

\theoremstyle{definition}
\newtheorem{exam}[thm]{Example}
\newtheorem{rem}[thm]{Remark}

\newcommand{\M}{\mathcal{M}}
\newcommand{\I}{\mathcal{I}}
\newcommand{\T}{\mathcal{T}}
\newcommand{\PP}{\mathfrak{P}}
\newcommand{\FF}{\mathfrak{F}}
\newcommand{\CC}{\mathfrak{C}}
\newcommand{\X}{\mathscr{X}}
\newcommand{\Y}{\mathscr{Y}}
\newcommand{\z}{\mathscr{Z}}
\newcommand{\A}{\mathscr{A}}
\newcommand{\B}{\mathscr{B}}
\newcommand{\C}{\mathscr{C}}
\newcommand{\D}{\mathscr{D}}
\newcommand{\e}{\mathscr{E}}
\newcommand{\f}{\mathscr{F}}
\newcommand{\g}{\mathscr{G}}
\newcommand{\h}{\mathscr{H}}
\newcommand{\YB}{2^{\mathscr{Y}}}
\newcommand{\ZB}{2^{\mathscr{Z}}}

\newcommand{\GG}{\widetilde{\g}}
\newcommand{\Z}{\mathbb{Z}}
\newcommand{\G}{\Gamma}
\newcommand{\GH}{\hat{\Gamma}}
\newcommand{\SL}{\textrm{SL}}
\newcommand{\GL}{\textrm{GL}}

\author[R. Kobayashi]{Ryoma Kobayashi}
\address[R. Kobayashi]{
Department of General Education,\endgraf
National Institute of Technology, Ishikawa College,\endgraf
Tsubata, Ishikawa, 929-0392, Japan
}
\email{kobayashi\_ryoma@ishikawa-nct.ac.jp}


\thanks{The author was supported by JSPS KAKENHI Grant Number JP19K14542 and JP22K13920.
}

\begin{document}

\title{The level $d$ mapping class group of a compact non-orientable surface}

\maketitle

%%%%%%%%%%%%%%%%%%%%%%%%%%%%%%%%%%%%%%%%%%%%%%%%%%%%%%%%%%%%%%%%%%%%%%%%%%%%%%%%%%%%%%%%%%%%%%%%%%%%
%%%%%%%%%%%%%%%%%%%%%%%%%%%%%%%%%%%%%%%%%%%%%%%%%%%%%%%%%%%%%%%%%%%%%%%%%%%%%%%%%%%%%%%%%%%%%%%%%%%%
%%%%%%%%%%%%%%%%%%%%%%%%%%%%%%%%%%%%%%%%%%%%%%%%%%%%%%%%%%%%%%%%%%%%%%%%%%%%%%%%%%%%%%%%%%%%%%%%%%%%
%%%%%%%%%%%%%%%%%%%%%%%%%%%%%%%%%%%%%%%%%%%%%%%%%%%%%%%%%%%%%%%%%%%%%%%%%%%%%%%%%%%%%%%%%%%%%%%%%%%%
%%%%%%%%%%%%%%%%%%%%%%%%%%%%%%%%%%%%%%%%%%%%%%%%%%%%%%%%%%%%%%%%%%%%%%%%%%%%%%%%%%%%%%%%%%%%%%%%%%%%
%%%%%%%%%%%%%%%%%%%%%%%%%%%%%%%%%%%%%%%%%%%%%%%%%%%%%%%%%%%%%%%%%%%%%%%%%%%%%%%%%%%%%%%%%%%%%%%%%%%%
%%%%%%%%%%%%%%%%%%%%%%%%%%%%%%%%%%%%%%%%%%%%%%%%%%%%%%%%%%%%%%%%%%%%%%%%%%%%%%%%%%%%%%%%%%%%%%%%%%%%
%%%%%%%%%%%%%%%%%%%%%%%%%%%%%%%%%%%%%%%%%%%%%%%%%%%%%%%%%%%%%%%%%%%%%%%%%%%%%%%%%%%%%%%%%%%%%%%%%%%%
%%%%%%%%%%%%%%%%%%%%%%%%%%%%%%%%%%%%%%%%%%%%%%%%%%%%%%%%%%%%%%%%%%%%%%%%%%%%%%%%%%%%%%%%%%%%%%%%%%%%
%%%%%%%%%%%%%%%%%%%%%%%%%%%%%%%%%%%%%%%%%%%%%%%%%%%%%%%%%%%%%%%%%%%%%%%%%%%%%%%%%%%%%%%%%%%%%%%%%%%%
\begin{abstract}
Let $N_{g,n}$ be a genus $g$ compact non-orientable surface with $n$ boundaries.
We explain about relations on the level $d$ mapping class group $\M_d(N_{g,0})$ of $N_{g,0}$ and the level $d$ principal congruence subgroup  $\G_d(g-1)$ of $\SL(g-1;\Z)$.
As applications, we give a normal generating set of $\M_d(N_{g,n})$ for $g\ge4$ and $n\ge0$, and finite generating sets of $\M_d(N_{g,n})$ for some $d$, any $g\ge4$ and $n\ge0$.
\end{abstract}

%%%%%%%%%%%%%%%%%%%%%%%%%%%%%%%%%%%%%%%%%%%%%%%%%%%%%%%%%%%%%%%%%%%%%%%%%%%%%%%%%%%%%%%%%%%%%%%%%%%%
%%%%%%%%%%%%%%%%%%%%%%%%%%%%%%%%%%%%%%%%%%%%%%%%%%%%%%%%%%%%%%%%%%%%%%%%%%%%%%%%%%%%%%%%%%%%%%%%%%%%
%%%%%%%%%%%%%%%%%%%%%%%%%%%%%%%%%%%%%%%%%%%%%%%%%%%%%%%%%%%%%%%%%%%%%%%%%%%%%%%%%%%%%%%%%%%%%%%%%%%%
%%%%%%%%%%%%%%%%%%%%%%%%%%%%%%%%%%%%%%%%%%%%%%%%%%%%%%%%%%%%%%%%%%%%%%%%%%%%%%%%%%%%%%%%%%%%%%%%%%%%
%%%%%%%%%%%%%%%%%%%%%%%%%%%%%%%%%%%%%%%%%%%%%%%%%%%%%%%%%%%%%%%%%%%%%%%%%%%%%%%%%%%%%%%%%%%%%%%%%%%%
%%%%%%%%%%%%%%%%%%%%%%%%%%%%%%%%%%%%%%%%%%%%%%%%%%%%%%%%%%%%%%%%%%%%%%%%%%%%%%%%%%%%%%%%%%%%%%%%%%%%
%%%%%%%%%%%%%%%%%%%%%%%%%%%%%%%%%%%%%%%%%%%%%%%%%%%%%%%%%%%%%%%%%%%%%%%%%%%%%%%%%%%%%%%%%%%%%%%%%%%%
%%%%%%%%%%%%%%%%%%%%%%%%%%%%%%%%%%%%%%%%%%%%%%%%%%%%%%%%%%%%%%%%%%%%%%%%%%%%%%%%%%%%%%%%%%%%%%%%%%%%
%%%%%%%%%%%%%%%%%%%%%%%%%%%%%%%%%%%%%%%%%%%%%%%%%%%%%%%%%%%%%%%%%%%%%%%%%%%%%%%%%%%%%%%%%%%%%%%%%%%%
%%%%%%%%%%%%%%%%%%%%%%%%%%%%%%%%%%%%%%%%%%%%%%%%%%%%%%%%%%%%%%%%%%%%%%%%%%%%%%%%%%%%%%%%%%%%%%%%%%%%
\section{Introduction}

Let $N=N_{g,n}$ be a genus $g\ge1$ compact non-orientable surface with $n\ge0$ boundaries.
Note that $N_{g,0}$ is a connected sum of $g$ real projective planes.
In this paper, we describe $N_{g,n}$ as shown in Figure~\ref{non-ori-surf}, that is, $N_{g,n}$ is a surface obtained by attaching $g$ M\"obius bands to a sphere removing $g+n$ disjoint open disks.
The \textit{mapping class group} $\M(N)$ of $N$ is the group consisting of isotopy classes of diffeomorphisms of $N$ fixing boundary pointwise.
$\M(N)$ is generated by \textit{Dehn twists} and \textit{crosscap slides} defined as follows (see \cite{Li1}).
For a simple closed curve $c$ of $N$, its regular neighborhood is either an annulus  or a M\"obius band.
We call $c$ a \textit{two sided} or a \textit{one sided} simple closed curve respectively.
For a two sided simple closed curve $c$, the \textit{Dehn twist} $t_c$ about $c$ is the isotopy class of the map acting as shown in Figure~\ref{dehn-slide}~(a).
The direction of the twist is indicated by an arrow written beside $c$ as shown in Figure~\ref{dehn-slide}~(a).
For a one sided simple closed curve $\mu$ of $N$ and an oriented two sided simple closed curve $\alpha$ of $N$ such that $N\setminus\alpha$ is non-orientable when genus of $N$ is more than or equal to $3$ and that $\mu$ and $\alpha$ intersect transversely at only one point, the \textit{crosscap slide} $Y_{\mu,\alpha}$ about $\mu$ and $\alpha$ is the isotopy class of the map described by pushing the crosscap which is a regular neighborhood of $\mu$ once along $\alpha$, as shown in Figure~\ref{dehn-slide}~(b).

%%%%%%%%%%%%%%%%%%%%%%%%%%%%%%%%%%%%%%%%%%%%%%%%%%%%%%%%%%%%%%%%%%%%%%%%%%%%%%%%%%%%%%%%%%%%%%%%%%%%
\begin{figure}[htbp]
\includegraphics{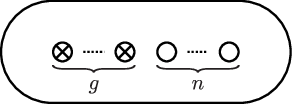}
\caption{A model of a compact non-orientable surface $N_{g,n}$.
The marks $\otimes$ are attached M\"obius bands and called \textit{crosscaps}.}\label{non-ori-surf}
\end{figure}
%%%%%%%%%%%%%%%%%%%%%%%%%%%%%%%%%%%%%%%%%%%%%%%%%%%%%%%%%%%%%%%%%%%%%%%%%%%%%%%%%%%%%%%%%%%%%%%%%%%%

%%%%%%%%%%%%%%%%%%%%%%%%%%%%%%%%%%%%%%%%%%%%%%%%%%%%%%%%%%%%%%%%%%%%%%%%%%%%%%%%%%%%%%%%%%%%%%%%%%%%
\begin{figure}[htbp]
\subfigure[The Dehn twist $t_c$ about $c$.]{\includegraphics{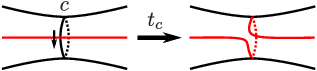}}\\
\subfigure[The crosscap slide $Y_{\mu,\alpha}$ about $\mu$ and $\alpha$.]{\includegraphics{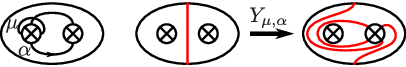}}
\caption{Descriptions of a Dehn twist and a crosscap slide.}\label{dehn-slide}
\end{figure}
%%%%%%%%%%%%%%%%%%%%%%%%%%%%%%%%%%%%%%%%%%%%%%%%%%%%%%%%%%%%%%%%%%%%%%%%%%%%%%%%%%%%%%%%%%%%%%%%%%%%

We now introduce some notations.
The \textit{twist subgroup} $\T(N)$ of $\M(N)$ is the subgroup of $\M(N)$ generated by all Dehn twists.
For $d\geq2$, the \textit{level $d$ mapping class group} $\M_d(N)$ of $N$ is the subgroup of $\M(N)$ acting trivially on $H_1(N;\Z/d\Z)$.
For $d\geq2$, the \textit{level $d$ twist subgroup} $\T_d(N)$ of $\M(N)$ is defined as $\T_d(N)=\M_d(N)\cap\T(N)$.
It is known that $\T(N)$ and $\T_2(N)$ are index $2$ subgroups of $\M(N)$ and $\M_2(N)$, respectively (see \cite{Li2, Sz2}), and $\T_d(N)=\M_d(N)$ for $d\geq3$. 
The \textit{Torelli group} $\I(N)$ of $N$ is the subgroup of $\M(N)$ acting trivially on $H_1(N;\Z)$.
We note that $\I(N)\subset\T_d(N)$.
For $d\geq2$, the \textit{level $d$ principal congruence subgroup} $\G_d(n)$ (resp. $\GH_d(n)$) of $\SL(n;\Z)$ (resp. $\GL(n;\Z)$) is the kernel of the natural homomorphism $\SL(n;\Z)\to\SL(n;\Z/d\Z)$ (resp. $\GL(n;\Z)\to\GL(n;\Z/d\Z)$).
By the definition, $\SL(n;\Z)$ and $\G_2(n)$ are index $2$ subgroups of $\GL(n;\Z)$ and $\GH_2(n)$, respectively, and $\G_d(n)=\GH_d(n)$ for $d\geq3$.

Here is an outline of this paper.
In Section~\ref{2}, we present results obtained for relations on $\M_d(N_{g,0})$ and $\G_d(g-1)$ for $d\geq3$.
As applications, we give a normal generating set of $\M_d(N_{g,n})$ for $g\ge4$ and $n\ge0$, in Section~\ref{3}.
In addition, in Sections~\ref{4}, we give a finite generating set of $\M_4(N_{g,0})$ for $g\ge4$, and explain about a finite generating set of $\M_{2^l}(N_{g,0})$ for $l\ge3$ and $g\ge4$.
Moreover, in Section~\ref{5}, we explain about a finite generating set of $\M_d(N_{g,n})$ for $n\ge1$.
As corollary, we can obtain explicit finite generating sets of $\M_2(N_{g,n})$, $\T_2(N_{g,n})$ and $\M_4(N_{g,n})$ for $g\ge4$ and $n\ge1$.
Note that finite generating sets for $\M_2(N_{g,0})$ and $\T_2(N_{g,0})$ are already known (see \cite{Sz2,HS,KO1}).

In this paper, the product $gf$ of mapping classes $f$ and $g$ means that we apply $f$ first and then $g$. 
In calculations, we use basic relations on $\M(N_{g,n})$.
For details, for instance see \cite{Sz1,O,KO2, IK1, KO3}.

%%%%%%%%%%%%%%%%%%%%%%%%%%%%%%%%%%%%%%%%%%%%%%%%%%%%%%%%%%%%%%%%%%%%%%%%%%%%%%%%%%%%%%%%%%%%%%%%%%%%
%%%%%%%%%%%%%%%%%%%%%%%%%%%%%%%%%%%%%%%%%%%%%%%%%%%%%%%%%%%%%%%%%%%%%%%%%%%%%%%%%%%%%%%%%%%%%%%%%%%%
%%%%%%%%%%%%%%%%%%%%%%%%%%%%%%%%%%%%%%%%%%%%%%%%%%%%%%%%%%%%%%%%%%%%%%%%%%%%%%%%%%%%%%%%%%%%%%%%%%%%
%%%%%%%%%%%%%%%%%%%%%%%%%%%%%%%%%%%%%%%%%%%%%%%%%%%%%%%%%%%%%%%%%%%%%%%%%%%%%%%%%%%%%%%%%%%%%%%%%%%%
%%%%%%%%%%%%%%%%%%%%%%%%%%%%%%%%%%%%%%%%%%%%%%%%%%%%%%%%%%%%%%%%%%%%%%%%%%%%%%%%%%%%%%%%%%%%%%%%%%%%
%%%%%%%%%%%%%%%%%%%%%%%%%%%%%%%%%%%%%%%%%%%%%%%%%%%%%%%%%%%%%%%%%%%%%%%%%%%%%%%%%%%%%%%%%%%%%%%%%%%%
%%%%%%%%%%%%%%%%%%%%%%%%%%%%%%%%%%%%%%%%%%%%%%%%%%%%%%%%%%%%%%%%%%%%%%%%%%%%%%%%%%%%%%%%%%%%%%%%%%%%
%%%%%%%%%%%%%%%%%%%%%%%%%%%%%%%%%%%%%%%%%%%%%%%%%%%%%%%%%%%%%%%%%%%%%%%%%%%%%%%%%%%%%%%%%%%%%%%%%%%%
%%%%%%%%%%%%%%%%%%%%%%%%%%%%%%%%%%%%%%%%%%%%%%%%%%%%%%%%%%%%%%%%%%%%%%%%%%%%%%%%%%%%%%%%%%%%%%%%%%%%
%%%%%%%%%%%%%%%%%%%%%%%%%%%%%%%%%%%%%%%%%%%%%%%%%%%%%%%%%%%%%%%%%%%%%%%%%%%%%%%%%%%%%%%%%%%%%%%%%%%%
\section{Relations on $\M_d(N_{g,0})$ and $\G_d(g-1)$}\label{2}

In this section, we explain about relations on $\M(N_{g,0})$ and $\GL(g-1;\Z)$.
For $1\leq{i_1}<i_2<\dots<i_k\leq{g}$, let $\alpha_{i_1,i_2,\dots,i_k}$ be an oriented simple closed curve of $N_{g,n}$ as shown in Figure~\ref{alpha}.
We regard $\alpha_{i_1,i_2,\dots,i_k}$ as an element of $H_1(N_{g,n};\Z)$.
Note that $\alpha_{i_1,i_2,\dots,i_k}=\alpha_{i_1}+\alpha_{i_2}+\cdots+\alpha_{i_k}$.
Then, as a $\Z$-module, we have presentations
\begin{eqnarray*}
H_1(N_{g,0};\Z)/{\langle{\alpha_{1,2,\dots,g}}\rangle}&=&\langle{\alpha_1,\alpha_2,\dots,\alpha_g\mid\alpha_{1,2,\dots,g}=0}\rangle\\
&=&\langle{\alpha_1,\alpha_2,\dots,\alpha_{g-1}}\rangle.
\end{eqnarray*}
It is known that any automorphism of $H_1(N_{g,0};\Z)$ preserves the homology class of $\alpha_{1,2,\dots,g}$ (see \cite{MP, HK}).
Hence $f\in\M(N_{g,0})$ induces the automorphism $f_\ast$ of $H_1(N_{g,0};\Z)/{\langle{\alpha_{1,2,\dots,g}}\rangle}$.
More precisely, for $f\in\M(N_{g,0})$, when $f(\alpha_j)=\sum_{i=1}^gc_{ij}\alpha_i$, we define $f_\ast(\alpha_j)=\sum_{i=1}^{g-1}(c_{ij}-c_{gj})\alpha_i$.
We can regard an automorphism of $H_1(N_{g,0};\Z)/{\langle{\alpha_{1,2,\dots,g}}\rangle}$ as a matrix in $GL(g-1,\Z)$, that is, $f_\ast=(c_{ij}-c_{gj})\in\GL(g-1;\Z)$ for $f\in\M(N_{g,0})$ above.
Let $\Phi:\M(N_{g,0})\to\GL(g-1;\Z)$ be the homomorphism defined as $\Phi(f)=f_\ast$.

%%%%%%%%%%%%%%%%%%%%%%%%%%%%%%%%%%%%%%%%%%%%%%%%%%%%%%%%%%%%%%%%%%%%%%%%%%%%%%%%%%%%%%%%%%%%%%%%%%%%
\begin{figure}[htbp]
\includegraphics{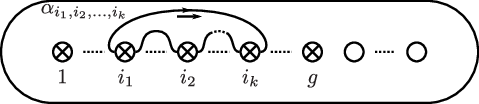}
\caption{An oriented simple closed curve $\alpha_{i_1,i_2,\dots,i_k}$ of $N_{g,n}$ for $1\leq{i_1}<i_2<\dots<i_k\leq{g}$.}\label{alpha}
\end{figure}
%%%%%%%%%%%%%%%%%%%%%%%%%%%%%%%%%%%%%%%%%%%%%%%%%%%%%%%%%%%%%%%%%%%%%%%%%%%%%%%%%%%%%%%%%%%%%%%%%%%%

%%%%%%%%%%%%%%%%%%%%%%%%%%%%%%%%%%%%%%%%%%%%%%%%%%%%%%%%%%%%%%%%%%%%%%%%%%%%%%%%%%%%%%%%%%%%%%%%%%%%
\begin{exam}\label{dehn-crosscap}
\begin{enumerate}
\item	Suppose that $k$ is even.
	We have
	$$
	t_{\alpha_{i_1,i_2,\dots,i_k}}^d(\alpha_j)=
	\left\{
	\begin{array}{ll}
	\alpha_j-d\alpha_{i_1,i_2,\dots,i_k}&(j=i_1,i_3,\dots,i_{k-1}),\\
	\alpha_j+d\alpha_{i_1,i_2,\dots,i_k}&(j=i_2,i_4,\dots,i_k),\\
	\alpha_j&(j\neq{}i_1,i_2,\dots,i_k).
	\end{array}
	\right.
	$$
	Hence $\Phi(t_{\alpha_{i,j}}^d)$ is the matrix whose $(i,j)$ entry is $d$, $(j,i)$ entry is $-d$, $(i,i)$ entry is $1-d$, $(j,j)$ entry is $1+d$, the other diagonal entries are $1$ and the other entries are $0$ if $j<g$, and $\Phi(t_{\alpha_{i,g}}^d)$ is the matrix whose $(j,i)$ entry is $d$ for $1\leq{j}\leq{g-1}$ with $j\ne{i}$, the diagonal entries are $1$ and the other entries are $0$.
	For example, when $g=3$, we have
	$$\Phi(t_{\alpha_{1,2}})=
	\begin{pmatrix}
	1-d&d\\
	-d&1+d
	\end{pmatrix}
	,~
	\Phi(t_{\alpha_{1,3}})=
	\begin{pmatrix}
	1&0\\
	d&1
	\end{pmatrix}
	,~
	\Phi(t_{\alpha_{2,3}})=
	\begin{pmatrix}
	1&d\\
	0&1
	\end{pmatrix}
	.$$
	When $g$ is even, $\Phi(t_{\alpha_{1,2,\dots,g}}^d)$ is the identity matrix.
\item	We have
	\begin{eqnarray*}
	Y_{\alpha_{i_1},\alpha_{i_1,i_2}}(\alpha_j)&=&
	\left\{
	\begin{array}{ll}
	-\alpha_{i_1}&(j=i_1),\\
	2\alpha_{i_1}+\alpha_{i_2}&(j=i_2),\\
	\alpha_j&(j\neq{}i_1,i_2),
	\end{array}
	\right.
	\\
	Y_{\alpha_{i_2},\alpha_{i_1,i_2}}(\alpha_j)&=&
	\left\{
	\begin{array}{ll}
	\alpha_{i_1}+2\alpha_{i_2}&(j=i_1),\\
	-\alpha_{i_2}&(j=i_2),\\
	\alpha_j&(j\neq{}i_1,i_2).
	\end{array}
	\right.
	\end{eqnarray*}
	Hence $\Phi(Y_{\alpha_i,\alpha_{i,j}})$ is the matrix whose $(i,j)$ entry is $2$, $(i,i)$ entry is $-1$, the other diagonal entries are $1$ and the other entries are $0$ if $j<g$, $\Phi(Y_{\alpha_j,\alpha_{i,j}})$ is the matrix whose $(j,i)$ entry is $2$, $(j,j)$ entry is $-1$, the other diagonal entries are $1$ and the other entries are $0$ if $j<g$, $\Phi(Y_{\alpha_i,\alpha_{i,g}})$ is the matrix whose $(i,i)$ entry is $-1$, the other diagonal entries are $1$ and the other entries are $0$, and $\Phi(Y_{\alpha_g,\alpha_{i,g}})$ is the matrix whose $(j,i)$ entry is $-2$ for $1\leq{j}\leq{g-1}$ with $j\neq{}i$, $(i,i)$ entry is $-1$, the other diagonal entries are $1$ and the other entries are $0$.
	For example, when $g=3$, we have
	$$
	\begin{array}{lll}
	\Phi(Y_{\alpha_1,\alpha_{1,2}})=
	\begin{pmatrix}
	-1&2\\
	0&1
	\end{pmatrix}
	,&
	\Phi(Y_{\alpha_2,\alpha_{1,2}})=
	\begin{pmatrix}
	1&0\\
	2&-1
	\end{pmatrix}
	,\\
	\Phi(Y_{\alpha_1,\alpha_{1,3}})=
	\begin{pmatrix}
	-1&0\\
	0&1
	\end{pmatrix}
	,&
	\Phi(Y_{\alpha_3,\alpha_{1,3}})=
	\begin{pmatrix}
	-1&0\\
	-2&1
	\end{pmatrix}
	,\\
	\Phi(Y_{\alpha_2,\alpha_{2,3}})=
	\begin{pmatrix}
	1&0\\
	0&-1
	\end{pmatrix}
	,&
	\Phi(Y_{\alpha_3,\alpha_{2,3}})=
	\begin{pmatrix}
	1&-2\\
	0&-1
	\end{pmatrix}
	.
	\end{array}
	$$
\end{enumerate}
\end{exam}
%%%%%%%%%%%%%%%%%%%%%%%%%%%%%%%%%%%%%%%%%%%%%%%%%%%%%%%%%%%%%%%%%%%%%%%%%%%%%%%%%%%%%%%%%%%%%%%%%%%%

%%%%%%%%%%%%%%%%%%%%%%%%%%%%%%%%%%%%%%%%%%%%%%%%%%%%%%%%%%%%%%%%%%%%%%%%%%%%%%%%%%%%%%%%%%%%%%%%%%%%
\begin{rem}\label{hom-pre}
Since $H_1(N_{g,0};\Z)$ has a presentation $\langle{\alpha_1,\alpha_2,\dots,\alpha_g\mid2\alpha_{1,2,\dots,g}=0}\rangle$, for any $x=\sum_{i=1}^gc_i\alpha_i$ and $y=\sum_{i=1}^gc^\prime_i\alpha_i\in{}H_1(N_{g,0};\Z)$, $x=y$ if and only if there is an even integer $l$ such that $c_i=c^\prime_i+l$ for $1\leq{i}\leq{g}$.
When $d$ is even, we denote any element of $H_1(N_{g,0};\Z)$ by $\sum_{i=1}^gc_i\alpha_i$, where $c_g$ is either $0$ or $1$.
Then the image with the natural projection $H_1(N_{g,0};\Z)\to{}H_1(N_{g,0};\Z/d\Z)$ can be presented by $\sum_{i=1}^{g-1}c^\prime_i[\alpha_i]+c_g[\alpha_g]$, where $0\leq{c^\prime_i}<d$ and $c^\prime_i\equiv{}c_i\pmod{d}$ for $1\leq{i}\leq{g-1}$.
When $d$ is odd, we denote any element of $H_1(N_{g,0};\Z)$ by $\sum_{i=1}^gc_i\alpha_i$, where $c_g$ is either $0$ or $d$.
Then the image with the natural projection $H_1(N_{g,0};\Z)\to{}H_1(N_{g,0};\Z/d\Z)$ can be presented by $\sum_{i=1}^{g-1}c^\prime_i[\alpha_i]$, where $0\leq{c^\prime_i}<d$ and $c^\prime_i\equiv{}c_i\pmod{d}$ for $1\leq{i}\leq{g-1}$.
\end{rem}
%%%%%%%%%%%%%%%%%%%%%%%%%%%%%%%%%%%%%%%%%%%%%%%%%%%%%%%%%%%%%%%%%%%%%%%%%%%%%%%%%%%%%%%%%%%%%%%%%%%%

It is known that 	$\Phi(\M_2(N_{g,0}))=\GH_2(g-1)$, $\Phi(\T_2(N_{g,0}))=\G_2(g-1)$ and $\ker\Phi|_{\M_2(N_{g,0})}=\ker\Phi|_{\T_2(N_{g,0})}=\I(N_{g,0})$ (see \cite{MP,HK}).
We now prove the following theorem.

%%%%%%%%%%%%%%%%%%%%%%%%%%%%%%%%%%%%%%%%%%%%%%%%%%%%%%%%%%%%%%%%%%%%%%%%%%%%%%%%%%%%%%%%%%%%%%%%%%%%
\begin{thm}\label{main-1}
For $g\geq4$ and $d\geq3$, we have followings.
\begin{enumerate}
\item	$\Phi(\M_d(N_{g,0}))=\G_d(g-1)$ if $d$ is even and $\Phi(\M_d(N_{g,0}))$ is a non-normal subgroup of $\G_d(g-1)$ if $d$ is odd.
\item	$\ker\Phi|_{\M_d(N_{g,0})}=\I(N_{g,0})$ if $g$ is odd or $d$ is even, and $\ker\Phi|_{\M_d(N_{g,0})}=\I(N_{g,0})\langle{t_{\alpha_{1,2,\dots,g}}^d}\rangle$ if $g$ is even and $d$ is odd.
\end{enumerate}
\end{thm}
%%%%%%%%%%%%%%%%%%%%%%%%%%%%%%%%%%%%%%%%%%%%%%%%%%%%%%%%%%%%%%%%%%%%%%%%%%%%%%%%%%%%%%%%%%%%%%%%%%%%

\begin{proof}
\begin{enumerate}
\item	For $f\in\M_d(N_{g,0})$, we can denote $f(\alpha_j)=\sum_{i=1}^gc_{ij}\alpha_i$, where $c_{jj}\equiv1\pmod{d}$ and $c_{ij}\equiv0\pmod{d}$ if $i\neq{j}$.
	Suppose that $c_{gj}$ is according to Remark~\ref{hom-pre}.
	For $1\leq{i,j}\leq{g-1}$, since $c_{jj}-c_{gj}\equiv1\pmod{d}$ and $c_{ij}-c_{gj}\equiv0\pmod{d}$ if $i\neq{j}$, we have $\Phi(f)\in\G_d(g-1)$.
	Hence $\Phi(\M_d(N_{g,0}))\subset\G_d(g-1)$.
	
	First, we show the case where $d$ is even.
	For $1\leq{i,j}\leq{g-1}$ with $i\neq{}j$, let $e_{ij}$ be the matrix whose $(i,j)$ entry and the diagonal entries are $1$ and the other entries are $0$.
	$\G_d(g-1)$ is generated by $e_{ij}^d$ and $e_{k1}e_{1k}^de_{k1}^{-1}$ for $1\leq{i,j}\leq{}g-1$ and $2\le{k}\le{g-1}$ with $i\neq{}j$, if $g\geq4$ (see \cite{IK2}).
	When $i<j$, we calculate
	\begin{eqnarray*}
	&&\Phi\left([t_{\alpha_{j,g}},Y_{\alpha_i,\alpha_{i,j}}]^\frac{d}{2}\right)=\left(e_{ij}^2\right)^\frac{d}{2}=e_{ij}^d,\\
	&&\Phi\left([t_{\alpha_{i,g}},Y_{\alpha_j,\alpha_{i,j}}]^\frac{d}{2}\right)=\left(e_{ji}^2\right)^\frac{d}{2}=e_{ji}^d,\\
	&&\Phi\left(t_{\alpha_{1,k}}^d\right)=\left(e_{k1}e_{1k}e_{k1}^{-1}\right)^d=e_{k1}e_{1k}^de_{k1}^{-1}
	\end{eqnarray*}
	(see Example~\ref{dehn-crosscap} or Figures~\ref{f}~(b) and (c)).
	Therefore the claim is obtained.
	
	Next, we show the case where $d$ is odd.
	It is known that any automorphism of $H_1(N_{g,0};\Z)$ induced from an element of $\M(N_{g,0})$ preserves the mod $2$ intersection form
	$$\star:H_1(N_{g,0};\Z)\times{}H_1(N_{g,0};\Z)\to\Z/2\Z$$
	(see \cite{MP,GP}).
	Suppose that there is $f\in\M(N_{g,0})$ such that $\Phi(f)=e_{12}^d$.
	If $f(\alpha_2)=d\alpha_1+\alpha_2$, since $f(\alpha_2)\star{}f(\alpha_2)\equiv{}d^2+1\equiv0\not\equiv\alpha_2\star\alpha_2\pmod2$, $f$ does not preserve $\star$.
	Hence $f(\alpha_2)=d\alpha_1+\alpha_2+\alpha_{1,2,\dots,g}$.
	If $f(\alpha_3)=\alpha_3$, since $f(\alpha_2)\star{}f(\alpha_3)\equiv1\not\equiv\alpha_2\star\alpha_3\pmod2$, $f$ does not preserve $\star$.
	Hence $f(\alpha_3)=\alpha_3+\alpha_{1,2,\dots,g}$.
	However, if $g$ is even, $f(\alpha_2)\star{}f(\alpha_2)\equiv0\not\equiv\alpha_2\star\alpha_2\pmod2$, and if $g$ is odd, $f(\alpha_3)\star{}f(\alpha_3)\equiv0\not\equiv\alpha_3\star\alpha_3\pmod2$, and so $f$ does not preserve $\star$.
	Hence there is no $f\in\M(N_{g,0})$, in particular no $f\in\M_d(N_{g,0})$, such that $\Phi(f)=e_{12}^d$, and so $\Phi(\M_d(N_{g,0}))\subsetneq\G_d(g-1)$.
	On the other hand, since $\Phi\left(t_{\alpha_{1,2}}^d\right)=e_{21}e_{12}^de_{21}^{-1}$ is conjugate to $e_{12}^d$, $\Phi(\M_d(N_{g,0}))$ is a non-normal subgroup of $\G_d(g-1)$.
	Therefore the claim is obtained.
\item	It is clear that $\ker\Phi|_{\M_d(N_{g,0})}\supset\I(N_{g,0})$.
	When $g$ is even, by Example~\ref{dehn-crosscap}, $t_{\alpha_{1,2,\dots,g}}^d$ is in $\ker\Phi|_{\M_d(N_{g,0})}$.
	Hence we have $\ker\Phi|_{\M_d(N_{g,0})}\supset\I(N_{g,0})\langle{t_{\alpha_{1,2,\dots,g}}^d}\rangle$ when $g$ is even.
	Note that for $f\in\ker\Phi|_{\M_d(N_{g,0})}$, the homology class $f(\alpha_i)$ is either $\alpha_i$ or $\alpha_i+\alpha_{1,2,\dots,g}$ if $d$ is even, and either $\alpha_i$ or $\alpha_i+d\alpha_{1,2,\dots,g}$ if $d$ is odd, by Remerk~\ref{hom-pre}.
	
	First, we show that the case where $g$ is odd or $d$ is even.
	When $g$ is odd, for $f\in\ker\Phi|_{\M_d(N_{g,0})}$, we see that $f(\alpha_i)\star{}f(\alpha_i)\equiv1$ or $0$ if $f(\alpha_i)=\alpha_i$ or $\alpha_i+\alpha_{1,2,\dots,g}$, respectively. 
	Hence $f(\alpha_i)\star{}f(\alpha_i)\equiv\alpha_i\star\alpha_i\pmod2$ if and only if $f(\alpha_i)=\alpha_i$, and so we have $f\in\I(N_{g,0})$.
	When $d$ is even, since $\I(N_{g,0})\subset\M_d(N_{g,0})\subset\M_2(N_{g,0})$, we have $\ker\Phi|_{\M_d(N_{g,0})}=\ker\Phi|_{\M_2(N_{g,0})}=\I(N_{g,0})$ by \cite{MP,HK}.
	Therefore the claim is obtained.
	
	Next, we show that the case where $g$ is even and $d$ is odd.
	Let $f\in\ker\Phi|_{\M_d(N_{g,0})}$.
	If $f(\alpha_i)=\alpha_i$ for any $1\le{i}\le{g}$, we have $f\in\I(N_{g,0})$.
	Suppose that there is $1\le{i}\le{g}$ such that $f(\alpha_i)=\alpha_i+d\alpha_{1,2,\dots,g}$.
	If there is $1\le{j}\le{g}$ with $j\ne{i}$ such that $f(\alpha_j)=\alpha_j$, since $f(\alpha_i)\star{}f(\alpha_j)\equiv1\not\equiv\alpha_i\star\alpha_j\pmod2$, $f$ does not preserve $\star$.
	Hence we have $f(\alpha_j)=\alpha_j+d\alpha_{1,2,\dots,g}$ for any $1\le{j}\le{g}$.
	Since the homology classes $f(\alpha_i)$ and $t_{\alpha_{1,2,\dots,g}}^d(\alpha_i)$ are equal for any $1\le{i}\le{g}$, we have $ft_{\alpha_{1,2,\dots,g}}^{-d}\in\I(N_{g,0})$, and hence $f\in\I(N_{g,0})\langle{t_{\alpha_{1,2,\dots,g}}^d}\rangle$.
	Therefore the claim is obtained.
\end{enumerate}
Thus we complete the proof.
\end{proof}

%%%%%%%%%%%%%%%%%%%%%%%%%%%%%%%%%%%%%%%%%%%%%%%%%%%%%%%%%%%%%%%%%%%%%%%%%%%%%%%%%%%%%%%%%%%%%%%%%%%%
\begin{rem}
By the argument similar to the proof of Theorem~\ref{main-1}, we have the followings.
\begin{enumerate}
\item	$\Phi(\M(N_{g,0}))$ and $\Phi(\T(N_{g,0}))$ are non-normal subgroups of $\GL(g-1)$ and $\SL(g-1)$ respectively.
\item	$\ker\Phi=\I(N_{g,0})$ or $\I(N_{g,0})\langle{t_{\alpha_{1,2,\dots,g}}}\rangle$ if $g$ is odd or even respectively.
\end{enumerate}
\end{rem}
%%%%%%%%%%%%%%%%%%%%%%%%%%%%%%%%%%%%%%%%%%%%%%%%%%%%%%%%%%%%%%%%%%%%%%%%%%%%%%%%%%%%%%%%%%%%%%%%%%%%

%%%%%%%%%%%%%%%%%%%%%%%%%%%%%%%%%%%%%%%%%%%%%%%%%%%%%%%%%%%%%%%%%%%%%%%%%%%%%%%%%%%%%%%%%%%%%%%%%%%%
%%%%%%%%%%%%%%%%%%%%%%%%%%%%%%%%%%%%%%%%%%%%%%%%%%%%%%%%%%%%%%%%%%%%%%%%%%%%%%%%%%%%%%%%%%%%%%%%%%%%
%%%%%%%%%%%%%%%%%%%%%%%%%%%%%%%%%%%%%%%%%%%%%%%%%%%%%%%%%%%%%%%%%%%%%%%%%%%%%%%%%%%%%%%%%%%%%%%%%%%%
%%%%%%%%%%%%%%%%%%%%%%%%%%%%%%%%%%%%%%%%%%%%%%%%%%%%%%%%%%%%%%%%%%%%%%%%%%%%%%%%%%%%%%%%%%%%%%%%%%%%
%%%%%%%%%%%%%%%%%%%%%%%%%%%%%%%%%%%%%%%%%%%%%%%%%%%%%%%%%%%%%%%%%%%%%%%%%%%%%%%%%%%%%%%%%%%%%%%%%%%%
%%%%%%%%%%%%%%%%%%%%%%%%%%%%%%%%%%%%%%%%%%%%%%%%%%%%%%%%%%%%%%%%%%%%%%%%%%%%%%%%%%%%%%%%%%%%%%%%%%%%
%%%%%%%%%%%%%%%%%%%%%%%%%%%%%%%%%%%%%%%%%%%%%%%%%%%%%%%%%%%%%%%%%%%%%%%%%%%%%%%%%%%%%%%%%%%%%%%%%%%%
%%%%%%%%%%%%%%%%%%%%%%%%%%%%%%%%%%%%%%%%%%%%%%%%%%%%%%%%%%%%%%%%%%%%%%%%%%%%%%%%%%%%%%%%%%%%%%%%%%%%
%%%%%%%%%%%%%%%%%%%%%%%%%%%%%%%%%%%%%%%%%%%%%%%%%%%%%%%%%%%%%%%%%%%%%%%%%%%%%%%%%%%%%%%%%%%%%%%%%%%%
%%%%%%%%%%%%%%%%%%%%%%%%%%%%%%%%%%%%%%%%%%%%%%%%%%%%%%%%%%%%%%%%%%%%%%%%%%%%%%%%%%%%%%%%%%%%%%%%%%%%
\section{A normal generating set for $\M_d(N_{g,n})$}\label{3}

The level $2$ mapping class group for an orientable closed surface can be generated by squares of Dehn twists about non-separating simple closed curves.
For $d\ge3$, the level $d$ mapping class group for an orientable closed surface can be generated by $d$-th powers of Dehn twists about non-separating simple closed curves and Dehn twists about separating simple closed curves (see \cite{Mc}).
Let $\beta_{i,j}$, $\gamma$, $\delta_k$, $\epsilon_{i,k}$, $\zeta_{k,l}$ and $\bar{\zeta}_{k,l}$ be simple closed curves as shown in Figure~\ref{beta-gamma-delta-epsilon-zeta}.
In this section, we prove the following theorem.

%%%%%%%%%%%%%%%%%%%%%%%%%%%%%%%%%%%%%%%%%%%%%%%%%%%%%%%%%%%%%%%%%%%%%%%%%%%%%%%%%%%%%%%%%%%%%%%%%%%%
\begin{thm}\label{main-2}
For $g\geq4$, $n\ge0$ and $d\geq2$, in $\M(N_{g,n})$, $\M_d(N_{g,n})$ is normally generated by
\begin{itemize}
\item	$Y_{\alpha_1,\alpha_{1,2}}$ only if $d=2$,
\item	$t_{\alpha_{1,2}}^d$ only if $d$ is odd or $n\ge1$ and $d\ne2$,
\item	$t_{\alpha_{1,2,3,4}}^d$ only if $d$ is odd and $g=4$,
\item	$\left(t_{\alpha_{1,2}}t_{\alpha^\prime_{1,2}}^{-1}\right)^\frac{d}{2}$ only if $d\ge4$ is even,
\item	$t_{\alpha_{1,2,3,4}}t_{\alpha^\prime_{1,2,3,4}}$, $t_{\beta_{1,2}}$ only if $d\ge3$,
\item	$t_\gamma$ only if $g=4$ and $d\ge3$,
\item	$t_{\delta_k}$, $t_{\epsilon_{g,k}}$ for $1\le{k}\le{n-1}$, only if $n\ge2$,
\item	$t_{\zeta_{k,l}}$ and $t_{\bar{\zeta}_{k,l}}$ for $1\le{k<l}\le{n-1}$, only if $n\ge3$,
\end{itemize}
where $\alpha^\prime_{1,2,3,4}=Y_{\alpha_2,\alpha_{1,2}}Y_{\alpha_3,\alpha_{3,4}}^{-1}(\alpha_{1,2,3,4})$ and $\alpha^\prime_{1,2}=Y_{\alpha_3,\alpha_{2,3}}(\alpha_{1,2})$.
\end{thm}
%%%%%%%%%%%%%%%%%%%%%%%%%%%%%%%%%%%%%%%%%%%%%%%%%%%%%%%%%%%%%%%%%%%%%%%%%%%%%%%%%%%%%%%%%%%%%%%%%%%%

%%%%%%%%%%%%%%%%%%%%%%%%%%%%%%%%%%%%%%%%%%%%%%%%%%%%%%%%%%%%%%%%%%%%%%%%%%%%%%%%%%%%%%%%%%%%%%%%%%%%
\begin{figure}[htbp]
\includegraphics{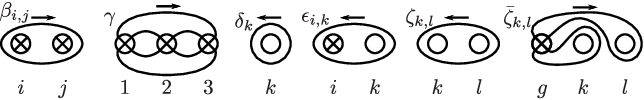}
\caption{Simple closede curves $\beta_{i,j}$, $\gamma$, $\delta_k$, $\epsilon_{i,k}$, $\zeta_{k,l}$ and $\bar{\zeta}_{k,l}$.}\label{beta-gamma-delta-epsilon-zeta}
\end{figure}
%%%%%%%%%%%%%%%%%%%%%%%%%%%%%%%%%%%%%%%%%%%%%%%%%%%%%%%%%%%%%%%%%%%%%%%%%%%%%%%%%%%%%%%%%%%%%%%%%%%%

%%%%%%%%%%%%%%%%%%%%%%%%%%%%%%%%%%%%%%%%%%%%%%%%%%%%%%%%%%%%%%%%%%%%%%%%%%%%%%%%%%%%%%%%%%%%%%%%%%%%
%%%%%%%%%%%%%%%%%%%%%%%%%%%%%%%%%%%%%%%%%%%%%%%%%%%%%%%%%%%%%%%%%%%%%%%%%%%%%%%%%%%%%%%%%%%%%%%%%%%%
%%%%%%%%%%%%%%%%%%%%%%%%%%%%%%%%%%%%%%%%%%%%%%%%%%%%%%%%%%%%%%%%%%%%%%%%%%%%%%%%%%%%%%%%%%%%%%%%%%%%
%%%%%%%%%%%%%%%%%%%%%%%%%%%%%%%%%%%%%%%%%%%%%%%%%%%%%%%%%%%%%%%%%%%%%%%%%%%%%%%%%%%%%%%%%%%%%%%%%%%%
%%%%%%%%%%%%%%%%%%%%%%%%%%%%%%%%%%%%%%%%%%%%%%%%%%%%%%%%%%%%%%%%%%%%%%%%%%%%%%%%%%%%%%%%%%%%%%%%%%%%
\subsection{Proof of Theorem~\ref{main-2} with $n=0$}\label{3.1}\

In this subsection, we prove Theorem~\ref{main-2} with $n=0$.
It is known that $\M_2(N_{g,0})$ is normally generated by only $Y_{\alpha_1,\alpha_{1,2}}$ (see \cite{Sz1}).
Hence we consider the case $d\ge3$.

First, we show the case where $d$ is even.
By Theorem~\ref{main-1}, we have the short exact sequence
$$1\to\I(N_{g,0})\to\M_d(N_{g,0})\to\G_d(g-1)\to1$$
if $d$ is even.
Hirose and the author \cite{HK} proved that for $g\ge4$, $\I(N_{g,0})$ is normally generated by $t_{\alpha_{1,2,3,4}}t_{\alpha^\prime_{1,2,3,4}}$, $t_{\beta_{1,2}}$ and $t_\gamma$ (only if $g=4$) in $\M(N_{g,0})$.
Hence, by the proof of Theorem~\ref{main-1}~(1), in $\M(N_{g,0})$, $\M_d(N_{g,0})$ is normally generated by $t_{\alpha_{1,2,3,4}}t_{\alpha^\prime_{1,2,3,4}}$, $t_{\beta_{1,2}}$, $t_\gamma$ (only if $g=4$), $[t_{\alpha_{j,g}},Y_{\alpha_i,\alpha_{i,j}}]^\frac{d}{2}$, $[t_{\alpha_{i,g}},Y_{\alpha_j,\alpha_{i,j}}]^\frac{d}{2}$ and $t_{\alpha_{1,k}}^d$ for $1\leq{i<j}\leq{}g-1$ and $2\le{k}\le{g-1}$.
We see
\begin{eqnarray*}
\left[t_{\alpha_{j,g}},Y_{\alpha_i,\alpha_{i,j}}\right]
&=&
t_{\alpha_{j,g}}Y_{\alpha_i,\alpha_{i,j}}t_{\alpha_{j,g}}^{-1}Y_{\alpha_i,\alpha_{i,j}}^{-1}=t_{\alpha_{j,g}}t_{Y_{\alpha_i,\alpha_{i,j}}(\alpha_{j,g})}^{-1},\\
\left[t_{\alpha_{i,g}},Y_{\alpha_j,\alpha_{i,j}}\right]
&=&
t_{\alpha_{i,g}}Y_{\alpha_j,\alpha_{i,j}}t_{\alpha_{i,g}}^{-1}Y_{\alpha_j,\alpha_{i,j}}^{-1}=t_{\alpha_{i,g}}t_{Y_{\alpha_j,\alpha_{i,j}}(\alpha_{i,g})}^{-1}.
\end{eqnarray*}
In addition, there are $f_1$, $f_2$ and $f_3\in\M(N_g)$ such that $f_1(\alpha_{1,2})=\alpha_{1,k}$, $f_2(\alpha_{1,2})=\alpha_{j,g}$, $f_2(\alpha^\prime_{1,2})=Y_{\alpha_i,\alpha_{i,j}}(\alpha_{j,g})$, $f_3(\alpha_{1,2})=\alpha_{i,g}$ and $f_3(\alpha^\prime_{1,2})=Y_{\alpha_j,\alpha_{i,j}}(\alpha_{i,g})$ (see Figures~\ref{f}~(a), (b) and (c)).
Hence we have
\begin{eqnarray*}
t_{\alpha_{1,k}}^d
&=&
t_{f_1(\alpha_{1,2})}^d=f_1t_{\alpha_{1,2}}^{\pm{}d}f_1^{-1},\\
\left(t_{\alpha_{j,g}}t_{Y_{\alpha_i,\alpha_{i,j}}(\alpha_{j,g})}^{-1}\right)^\frac{d}{2}
&=&
\left(t_{f_2(\alpha_{1,2})}t_{f_2(\alpha^\prime_{1,2})}^{-1}\right)^\frac{d}{2}=f_2\left(t_{\alpha_{1,2}}t_{\alpha^\prime_{1,2}}^{-1}\right)^{\pm\frac{d}{2}}f_2^{-1},\\
\left(t_{\alpha_{i,g}}t_{Y_{\alpha_j,\alpha_{i,j}}(\alpha_{i,g})}^{-1}\right)^\frac{d}{2}
&=&
\left(t_{f_3(\alpha_{1,2})}t_{f_3(\alpha^\prime_{1,2})}^{-1}\right)^\frac{d}{2}=f_3\left(t_{\alpha_{1,2}}t_{\alpha^\prime_{1,2}}^{-1}\right)^{\pm\frac{d}{2}}f_3^{-1}.
\end{eqnarray*}
Moreover, we have
$$t_{\alpha_{1,2}}^d=\left(t_{\alpha_{1,2}}t_{\alpha^\prime_{1,2}}^{-1}\right)^\frac{d}{2}\left(t_{\alpha_{1,2;3}}t_{\alpha_{1,2;4}}^{-1}\right)^\frac{d}{2}\left(t_{\alpha_{1,2;4}}t_{\alpha_{1,2;5}}^{-1}\right)^\frac{d}{2}\cdots\left(t_{\alpha_{1,2;g-1}}t_{\alpha_{1,2;g}}^{-1}\right)^\frac{d}{2},$$
where $\alpha_{1,2;l}$ is a simple closed curve as shown in Figure~\ref{f}~(d), and there is $f_l\in\M(N_g)$ such that $f_l(\alpha_{1,2})=\alpha_{1,2;l-1}$ and $f_l(\alpha^\prime_{1,2})=\alpha_{1,2;l}$ for $4\le{l}\le{g}$ (see Figure~\ref{f}~(d)).
Hence we have
$$\left(t_{\alpha_{1,2;l-1}}t_{\alpha_{1,2;l}}^{-1}\right)^\frac{d}{2}=\left(t_{f_l(\alpha_{1,2})}t_{f_l(\alpha^\prime_{1,2})}^{-1}\right)^\frac{d}{2}=f_l\left(t_{\alpha_{1,2}}t_{\alpha^\prime_{1,2}}^{-1}\right)^{\pm\frac{d}{2}}f_l^{-1}.$$
Therefore we obtain the claim.

%%%%%%%%%%%%%%%%%%%%%%%%%%%%%%%%%%%%%%%%%%%%%%%%%%%%%%%%%%%%%%%%%%%%%%%%%%%%%%%%%%%%%%%%%%%%%%%%%%%%
\begin{figure}[htbp]
\subfigure[$f_1(\alpha_{1,2})=\alpha_{1,k}$.]{\includegraphics{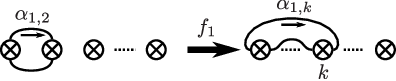}}\\
\subfigure[$f_2(\alpha_{1,2})=\alpha_{j,g}$ and $f_2(\alpha^\prime_{1,2})=Y_{\alpha_i,\alpha_{i,j}}(\alpha_{j,g})$.]{\includegraphics{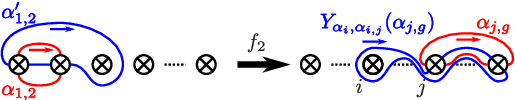}}\\
\subfigure[$f_3(\alpha_{1,2})=\alpha_{i,g}$ and $f_3(\alpha^\prime_{1,2})=Y_{\alpha_j,\alpha_{i,j}}(\alpha_{i,g})$.]{\includegraphics{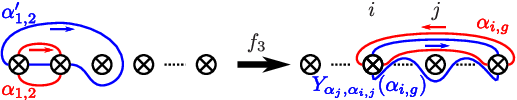}}\\
\subfigure[$f_l(\alpha_{1,2})=\alpha_{1,2;l-1}$ and $f_l(\alpha^\prime_{1,2})=\alpha_{1,2;l}$.
Note that $t_{\alpha_{1,2;3}}=t_{\alpha^\prime_{1,2}}$ and $t_{\alpha_{1,2;g}}^{-1}=t_{\alpha_{1,2}}$.]{\includegraphics{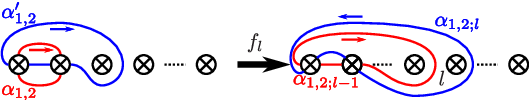}}
\caption{}\label{f}
\end{figure}
%%%%%%%%%%%%%%%%%%%%%%%%%%%%%%%%%%%%%%%%%%%%%%%%%%%%%%%%%%%%%%%%%%%%%%%%%%%%%%%%%%%%%%%%%%%%%%%%%%%%

Next, we show the case where $d$ is odd.
We note that $H_1(N_{g,0};\Z/2\Z)$ is isomorphic to $\left(\Z/2\Z\right)^{g}$.
Let $O_2(g)$ denote the orthogonal group on $\Z/2\Z$ of rank $g$.
The natural action on $H_1(N_{g,0};\Z/2\Z)$ by $\M(N_{g,0})$ induces the epimorphism $\Psi:\M(N_{g,0})\to{}O_2(g)$ whose kernel is $\M_2(N_{g,0})$ (see \cite{GP}).
Since $\M(N_{g,0})$ is generated by $t_{\alpha_{i,i+1}}$, $t_{\alpha_{1,2,3,4}}$ and $Y_{\alpha_1,\alpha_{1,2}}$ for $1\le{i}\le{g-1}$ (see \cite{Sz2}), $O_2(g)$ is generated by $\Psi(t_{\alpha_{i,i+1}})$ and $\Psi(t_{\alpha_{1,2,3,4}})$ for $1\le{i}\le{g-1}$.
In addition, since $\Psi(t_{\alpha_{i,i+1}}^d)=\Psi(t_{\alpha_{i,i+1}})$ and $\Psi(t_{\alpha_{1,2,3,4}}^d)=\Psi(t_{\alpha_{1,2,3,4}})$, we have $\Psi(\M_d(N_{g,0}))=O_2(g)$.
Moreover, since $\M_d(N_{g,0})\cap\M_2(N_{g,0})=\M_{2d}(N_{g,0})$, we have the short exact sequence
$$1\to\M_{2d}(N_{g,0})\to\M_d(N_{g,0})\to{}O_2(g)\to1.$$
Hence $\M_d(N_{g,0})$ is normally generated by $t_{\alpha_{1,2,3,4}}t_{\alpha^\prime_{1,2,3,4}}$, $t_{\beta_{1,2}}$, $t_\gamma$ (only if $g=4$), $\left(t_{\alpha_{1,2}}t_{\alpha^\prime_{1,2}}^{-1}\right)^d$, $t_{\alpha_{i,i+1}}^d$ and $t_{\alpha_{1,2,3,4}}^d$ for $1\le{i}\le{g-1}$.
There is $f_i\in\M(N_{g,0})$ such that $f_i(\alpha_{1,2})=\alpha_{i,i+1}$ and $f_1(\alpha_{1,2})=\alpha_{1,2,3,4}$ (if $g\ge5$) for $2\le{i}\le{g-1}$.
Hence we have
\begin{eqnarray*}
\left(t_{\alpha_{1,2}}t_{\alpha^\prime_{1,2}}^{-1}\right)^d
&=&t_{\alpha_{1,2}}^dt_{\alpha^\prime_{1,2}}^{-d}=t_{\alpha_{1,2}}^d(Y_{\alpha_3,\alpha_{2,3}}t_{\alpha_{1,2}}^{-d}Y_{\alpha_3,\alpha_{2,3}}^{-1}),\\
t_{\alpha_{i,i+1}}^d
&=&t_{f_i(\alpha_{1,2})}^d=f_it_{\alpha_{1,2}}^{\pm{}d}f_i^{-1},\\
t_{\alpha_{1,2,3,4}}^d
&=&t_{f_1(\alpha_{1,2})}^d=f_1t_{\alpha_{1,2}}^{\pm{}d}f_1^{-1}.
\end{eqnarray*}
Note that there is no $f\in\M(N_{4,0})$ such that $f(\alpha_{1,2})=\alpha_{1,2,3,4}$, more generally,  there is no $f\in\M(N_{g,0})$ such that $f(\alpha_{1,2})=\alpha_{1,2,\dots,g}$ for $g\ge3$.
Therefore we obtain the claim.

Thus we complete the proof.

%%%%%%%%%%%%%%%%%%%%%%%%%%%%%%%%%%%%%%%%%%%%%%%%%%%%%%%%%%%%%%%%%%%%%%%%%%%%%%%%%%%%%%%%%%%%%%%%%%%%
\begin{rem}
$\T_2(N_{g,0})$ can be normally generated by $t_{\alpha_{1,2}}t_{\alpha^\prime_{1,2}}^{-1}$ in $\M(N_{g,0})$, for $g=3$ or $g\ge5$, and $\T_2(N_{4,0})$ can be normally generated by $t_{\alpha_{1,2}}t_{\alpha^\prime_{1,2}}^{-1}$ and $t_{\alpha_{1,2,3,4}}^2$ in $\M(N_{4,0})$ (see \cite{KO1}).
In addition, $\T_2(N_{g,0})$ can not be generated by only $2$nd powers of Dehn twists about non-separating simple closed curves and Dehn twists about separating simple closed curves for $g\ge4$ (see \cite{IK1}).
We do not know whether or not $\M_d(N_{g,0})$ can be normally generated by $t_{\alpha_{1,2}}^d$, $t_{\alpha_{1,2,3,4}}^d$ (only if $g=4$), $t_{\alpha_{1,2,3,4}}t_{\alpha^\prime_{1,2,3,4}}$, $t_{\beta_{1,2}}$ and $t_\gamma$ (only if $g=4$) in $\M(N_{g,0})$, for even $d\ge4$.
\end{rem}
%%%%%%%%%%%%%%%%%%%%%%%%%%%%%%%%%%%%%%%%%%%%%%%%%%%%%%%%%%%%%%%%%%%%%%%%%%%%%%%%%%%%%%%%%%%%%%%%%%%%

%%%%%%%%%%%%%%%%%%%%%%%%%%%%%%%%%%%%%%%%%%%%%%%%%%%%%%%%%%%%%%%%%%%%%%%%%%%%%%%%%%%%%%%%%%%%%%%%%%%%
%%%%%%%%%%%%%%%%%%%%%%%%%%%%%%%%%%%%%%%%%%%%%%%%%%%%%%%%%%%%%%%%%%%%%%%%%%%%%%%%%%%%%%%%%%%%%%%%%%%%
%%%%%%%%%%%%%%%%%%%%%%%%%%%%%%%%%%%%%%%%%%%%%%%%%%%%%%%%%%%%%%%%%%%%%%%%%%%%%%%%%%%%%%%%%%%%%%%%%%%%
%%%%%%%%%%%%%%%%%%%%%%%%%%%%%%%%%%%%%%%%%%%%%%%%%%%%%%%%%%%%%%%%%%%%%%%%%%%%%%%%%%%%%%%%%%%%%%%%%%%%
%%%%%%%%%%%%%%%%%%%%%%%%%%%%%%%%%%%%%%%%%%%%%%%%%%%%%%%%%%%%%%%%%%%%%%%%%%%%%%%%%%%%%%%%%%%%%%%%%%%%
\subsection{Proof of Theorem~\ref{main-2} with $n\ge1$}\label{3.2}\

In this subsection, we prove Theorem~\ref{main-2} with $n\ge1$.

Put a interier point $\ast\in{}N_{g,n-1}$.
Let $\M(N_{g,n-1},\ast)$ denote the group consisting of isotopy classes of diffeomorphisms of $N_{g,n-1}$ fixing $\ast$ and boundary point wise.
$\I(N_{g,n-1},\ast)$ and $\M_d(N_{g,n-1},\ast)$ are the subgroups of $\M(N_{g,n-1},\ast)$ acting trivially on $H_1(N_{g,n-1};\Z)$ and $H_1(N_{g,n-1};\Z/d\Z)$, respectively.
We regard $N_{g,n}$ as a surface obtained by removing an open disk neighborhood of $\ast$ from $N_{g,n-1}$.
Then we can define homomorphisms
\begin{eqnarray*}
&&\PP:\pi_1(N_{g,n-1},\ast)\to\M(N_{g,n-1},\ast),\\
&&\FF:\M(N_{g,n-1},\ast)\to\M(N_{g,{n-1}}),\\
&&\CC:\M(N_{g,n})\to\M(N_{g,n-1},\ast),
\end{eqnarray*}
called the \textit{point pushing map}, the \textit{forgetful map} and the \textit{capping map}, respectively (for detail, for instance see \cite{FM, Ko, KO2, IK1, KO3}).

$\pi_1(N_{g,n-1},\ast)$ is generated by $x_i$ and $y_k$ for $1\le{i}\le{g}$ and $1\le{k}\le{n-1}$, where $x_i$ and $y_k$ are oriented simple loops based at $\ast$ as shown in Figure.
Let $\pi_1^+(N_{g,n-1},\ast)$ be the subgroup of $\pi_1(N_{g,n-1},\ast)$ generated by two sided loops.
It is known that $\pi_1^+(N_{g,n-1},\ast)$ is an index $2$ subgroup of $\pi_1(N_{g,n-1},\ast)$ and is generated by $x_ix_g$, $x_i^2$, $y_k$ and $z_k=x_gy_kx_g^{-1}$ for $1\le{i}\le{g}$ and $1\le{k}\le{n-1}$ (see \cite{Ko,KO2}).

%%%%%%%%%%%%%%%%%%%%%%%%%%%%%%%%%%%%%%%%%%%%%%%%%%%%%%%%%%%%%%%%%%%%%%%%%%%%%%%%%%%%%%%%%%%%%%%%%%%%
\begin{figure}[htbp]
\includegraphics{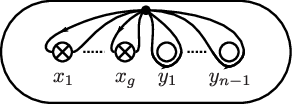}
\caption{Generators $x_1$, $\dots$, $x_g$ and $y_1$, $\dots$, $y_{n-1}$ of $\pi_1(N_{g,n-1},\ast)$.}\label{pi_1}
\end{figure}
%%%%%%%%%%%%%%%%%%%%%%%%%%%%%%%%%%%%%%%%%%%%%%%%%%%%%%%%%%%%%%%%%%%%%%%%%%%%%%%%%%%%%%%%%%%%%%%%%%%%

For $x\in\pi_1(N_{g,n-1},\ast)$, $\PP(x)$ is in $\CC(\M(N_{g,n}))$ if and only if $x$ is in $\pi_1^+(N_{g,n-1},\ast)$ (see \cite{Ko}).
For $x\in\pi_1^+(N_{g,n-1},\ast)$, let $x_\star$ be the automorphism of $H_1(N_{g,n};\Z/d\Z)$ induced from any lift by $\CC$ of $\PP(x)$.
Note that $\PP(x)$ is in $\I(N_{g,n-1},\ast)$,  and so $\M_d(N_{g,n-1},\ast)$.
Hence for $1\le{i}\le{g}$, there is $m_i\in\Z/d\Z$ such that $x_\star(\alpha_i)=\alpha_i+m_i\delta_n$.
Then we can define the homomorphism $\theta:\pi_1^+(N_{g,n-1},\ast)\to(\Z/d\Z)^g$ as $\theta(x)=(m_1,m_2,\dots,m_g)$.

First, we prove the following proposition.

\begin{prop}\label{PFC}
The sequence
$$\ker\theta\to\CC(\M_d(N_{g,n}))\to\M_d(N_{g,n-1})\to1$$
is exact.
\end{prop}

\begin{proof}
First, we show $\FF(\CC(\M_d(N_{g,n})))=\M_d(N_{g,n-1})$.
For simplicity, we regard any oriented simple closed curve of $N_{g,n}$ as an element of $H_1(N_{g,n};\Z/d\Z)$.
For any $f\in\M(N_{g,n})$, let $f_\star$ be the automorphism of $H_1(N_{g,n};\Z/d\Z)$ induced from $f$.
It follows that $f_\star(a_{1,2,\dots,g})=\alpha_{1,2,\dots,g}$ and $f_\star(\delta_k)=\delta_k$.
In fact, for any generator $f$ of $\M(N_{g,n})$ defined in \cite{KO2}, we can check that $f_\star(a_{1,2,\dots,g})=\alpha_{1,2,\dots,g}$.
It is clear that $\FF(\CC(\M_d(N_{g,n})))\subset\M_d(N_{g,n-1})$.
For any $f\in\M_d(N_{g,n-1})$, there is $h\in\M(N_{g,n})$ such that $\FF(\CC(h))=f$.
Suppose that $h_\star(\alpha_i)=\alpha_i+n_i\delta_n$ for $1\le{i}\le{g}$, where $n_i\in\Z/d\Z$.
Let $\tau_i$ be any lift by $\CC$ of $\PP(x_ix_g)$ for $1\le{i}\le{g-1}$ and let $\tau=\tau_{g-1}^{n_{g-1}}\cdots\tau_2^{n_2}\tau_1^{n_1}$.
Note that $\FF(\CC(\tau_i))=\FF(\PP(x_ix_g))=1$, and hence $\FF(\CC(\tau))=1$.
We see $(\tau_i)_\star(\alpha_i)=\alpha_i-\delta_n$, $(\tau_i)_\star(\alpha_g)=\alpha_g+\delta_n$ and $(\tau_i)_\star(\alpha_j)=\alpha_j$ for $1\le{j}\le{g-1}$ with $j\neq{}i$.
Hence we calculate
\begin{eqnarray*}
(\tau{}h)_\star(\alpha_i)
&=&\tau_\star(h_\star(\alpha_i))\\
&=&\tau_\star(\alpha_i+n_i\delta_n)\\
&=&\tau_\star(\alpha_i)+n_i\tau_\star(\delta_n)\\
&=&(\tau_i^{n_i})_\star(\alpha_i)+n_i\delta_n\\
&=&(\alpha_i-n_i\delta_n)+n_i\delta_n\\
&=&\alpha_i
\end{eqnarray*}
for $1\le{i}\le{g-1}$, and
\begin{eqnarray*}
(\tau{}h)_\star(\alpha_g)
&=&\tau_\star(h_\star(\alpha_g))\\
&=&\tau_\star(\alpha_g+n_g\delta_n)\\
&=&\tau_\star(\alpha_g)+n_g\tau_\star(\delta_n)\\
&=&\left(\alpha_g+\sum_{i=1}^{g-1}n_i\delta_n\right)+n_g\delta_n\\
&=&\alpha_g+\sum_{i=1}^gn_i\delta_n.
\end{eqnarray*}
In addition, since $(\tau{}h)_\star(\alpha_{1,2,\dots,g})=\alpha_{1,2,\dots,g}$, we calculate
\begin{eqnarray*}
(\tau{}h)_\star(\alpha_{1,2,\dots,g})
&=&\sum_{i=1}^g(\tau{}h)_\star(\alpha_i)\\
&=&\sum_{i=1}^{g-1}\alpha_i+\left(\alpha_g+\sum_{i=1}^gn_i\delta_n\right)\\
&=&\alpha_{1,2,\dots,g}+\sum_{i=1}^gn_i\delta_n\\
&=&\alpha_{1,2,\dots,g},
\end{eqnarray*}
and hence $\displaystyle\sum_{i=1}^gn_i=0$.
Therefore $(\tau{}h)_\star$ is identity, that is, $\tau{}h$ is in $\M_d(N_{g,n})$.
Since $\FF(\CC(\tau{}h))=\FF(\CC(\tau))\FF(\CC(h))=f$, we conclude that $\FF(\CC(\M_d(N_{g,n})))\supset\M_d(N_{g,n-1})$, and so the claim is obtained.

Next, we show $\PP(\ker\theta)=\ker\FF\cap\CC(\M_d(N_{g,n}))$.
We note that the sequence
$$\pi_1(N_{g,n-1},\ast)\overset{\PP}{\to}\M(N_{g,n-1},\ast)\overset{\FF}{\to}\M(N_{g,n-1})\to1$$
is exact (see \cite{Ko}).
For any $x\in\ker\theta$, since $x_\star=1$, we have that $\PP(x)$ is in $\CC(\M_d(N_{g,n}))$. 
In addition, since $\PP(x)$ is in $\ker\FF$, it follows that $\PP(x)$ is in $\ker\FF\cap\CC(\M_d(N_{g,n}))$, and so $\PP(\ker\theta)\subset\ker\FF\cap\CC(\M_d(N_{g,n}))$.
For any $f\in\ker\FF\cap\CC(\M_d(N_{g,n}))$, since $f$ is in $\ker\FF$, there is $x\in\pi_1(N_{g,n-1},\ast)$ such that $\PP(x)=f$.
In addition, since $f$ is in $\CC(\M_d(N_{g,n}))$, we have $x_\star=1$, and so $x$ is in $\ker\theta$.
Hence it follows that $f$ is in $\PP(\ker\theta)$, and so $\PP(\ker\theta)\supset\ker\FF\cap\CC(\M_d(N_{g,n}))$.
Therefore the claim is obtained.
\end{proof}

Next, we prove the following proposition.

\begin{prop}\label{ker-theta}
In $\pi_1^+(N_{g,n-1},\ast)$, $\ker\theta$ is normally generated by $x_i^2$, $y_k$, $z_k$, $(x_ix_jx_g)^2$ and $(x_ix_g)^d$ for $1\le{i}\le{g}$, $i<j<g$ and $1\le{k}\le{n-1}$.
\end{prop}

\begin{proof}
For $x\in\pi_1^+(N_{g,n-1},\ast)$, let $x_\ast$ be the automorphism of $H_1(N_{g,n};\Z)$ induced from any lift by $\CC$ of $\PP(x)$.
For $1\le{i}\le{g}$, there is $m_i\in\Z$ such that $x_\ast(\alpha_i)=\alpha_i+m_i\delta_n$.
Then we can define the homomorphism $\bar{\theta}:\pi_1^+(N_{g,n-1},\ast)\to\Z^g$ as $\bar{\theta}(x)=(m_1,m_2,\dots,m_g)$.
Note that $\theta=\pi\circ\bar{\theta}$, where $\pi:\Z^g\to(\Z/d\Z)^g$ is the natural projection.
The author \cite{Ko} showed that $\textrm{Im}\bar{\theta}$ is isomorphic to $\Z^{g-1}$ which is generated by $\bar{\theta}(x_ix_g)$ for $1\le{i}\le{g-1}$.
Hence $\textrm{Im}\theta$ is isomorphic to $(\Z/d\Z)^{g-1}$ which is generated by $\theta(x_ix_g)$ and $\bar{\theta}(\ker\theta)=\ker\pi\cap\textrm{Im}\bar{\theta}$ is isomorphic to $(d\Z)^{g-1}$ which is generated by $\bar{\theta}((x_ix_g)^d)$, for $1\le{i}\le{g-1}$.
In addition, the author \cite{Ko} also showed that in $\pi_1^+(N_{g,n-1},\ast)$, $\ker\bar{\theta}$ is normally generated by $x_i^2$, $y_k$, $z_k$ and $(x_ix_jx_g)^2$ for $1\le{i}\le{g}$, $i<j<g$ and $1\le{k}\le{n-1}$.
Therefore, since $\ker\theta/\ker\bar{\theta}$ is isomorphic to $\bar{\theta}(\ker\theta)$, we conclude that in $\pi_1^+(N_{g,n-1},\ast)$, $\ker\theta$ is normally generated by $x_i^2$, $y_k$, $z_k$, $(x_ix_jx_g)^2$ and $(x_ix_g)^d$ for $1\le{i}\le{g}$, $i<j<g$ and $1\le{k}\le{n-1}$.
Thus the claim is obtained.
\end{proof}

We now prove Theorem~\ref{main-2} with $n\ge1$.

Let $\eta_{i,j;k}$ and $\alpha_{i;k}$ be simple closed curves as shown in Figure~\ref{ez}.
By Propositions~\ref{PFC} and \ref{ker-theta}, in $\M(N_{g,n-1},\ast)$, $\CC(\M_d(N_{g,n}))$ is normally generated by any lifts by $\FF$ of any normal generator of $\M_d(N_{g,n-1})$ in $\M(N_{g,n-1})$ and $\PP(x_i^2)$, $\PP(y_k)$, $\PP(z_k)$, $\PP((x_ix_jx_g)^2)$ and $\PP((x_ix_g)^d)$ for $1\le{i}\le{g}$, $i<j<g$ and $1\le{k}\le{n-1}$.
In addition, since $\ker\CC$ is generated by $t_{\delta_n}$, in $\M(N_{g,n})$, $\M_d(N_{g,n})$ is normally generated by any lifts by $\CC$ of any normal generator of $\CC(\M_d(N_{g,n}))$ in $\M(N_{g,n-1},\ast)$ and $t_{\delta_n}$.
We see that $\CC(t_{\epsilon_{i,n}})=\PP(x_i^2)$, $\CC(t_{\zeta_{k,n}}t_{\delta_k}^{-1})=\PP(y_k)$, $\CC(t_{\bar{\zeta}_{k,n}}t_{\delta_k}^{-1})=\PP(z_k)$, $\CC(t_{\eta_{i,j;n}})=\PP((x_ix_jx_g)^2)$ and $\CC(t_{\alpha_{i,g}}^{-d}t_{\alpha_{i;n}}^d)=\PP((x_ix_g)^d)$.
Note that $t_{\epsilon_{i,n}}$ and $t_{\eta_{i,j;n}}$ are conjugate to $t_{\epsilon_{g,n}}$, and that $t_{\alpha_{i,g}}^d$ and $t_{\alpha_{i;n}}^d$ are conjugate to $t_{\alpha_{1,2}}^d$.
Therefore, by induction on $n$, we have that in $\M(N_{g,n})$, $\M_d(N_{g,n})$ is normally generated by
\begin{itemize}
\item	$Y_{\alpha_1,\alpha_{1,2}}$ only if $d=2$,
\item	$t_{\alpha_{1,2}}^d$,
\item	$t_{\alpha_{1,2,3,4}}^d$ only if $d$ is odd and $g=4$,
\item	$\left(t_{\alpha_{1,2}}t_{\alpha^\prime_{1,2}}^{-1}\right)^\frac{d}{2}$ only if $d$ is even,
\item	$t_{\alpha_{1,2,3,4}}t_{\alpha^\prime_{1,2,3,4}}$, $t_{\beta_{1,2}}$ only if $d\ge3$,
\item	$t_\gamma$ only if $g=4$ and $d\ge3$,
\item	$t_{\delta_k}$, $t_{\epsilon_{g,k}}$ for $1\le{k}\le{n}$,
\item	$t_{\zeta_{k,l}}$ and $t_{\bar{\zeta}_{k,l}}$ for $1\le{k<l}\le{n}$, only if $n\ge2$.
\end{itemize}
It is known that $t_{\alpha_{1,2}}^2$ is a product of crosscap slides (see \cite{Sz1}).
In addition, the author showed that $t_{\delta_n}$, $t_{\epsilon_{g,n}}$, $t_{\zeta_{k,n}}$ and $t_{\bar{\zeta}_{k,n}}$ are products of conjugates of other normal generators (see \cite{Ko}).

%%%%%%%%%%%%%%%%%%%%%%%%%%%%%%%%%%%%%%%%%%%%%%%%%%%%%%%%%%%%%%%%%%%%%%%%%%%%%%%%%%%%%%%%%%%%%%%%%%%%
\begin{figure}[htbp]
\includegraphics{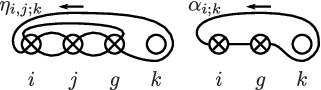}
\caption{Simple closed curves $\eta_{i,j;k}$ and $\alpha_{i;k}$.}\label{ez}
\end{figure}
%%%%%%%%%%%%%%%%%%%%%%%%%%%%%%%%%%%%%%%%%%%%%%%%%%%%%%%%%%%%%%%%%%%%%%%%%%%%%%%%%%%%%%%%%%%%%%%%%%%%

Thus we complete the proof.

%%%%%%%%%%%%%%%%%%%%%%%%%%%%%%%%%%%%%%%%%%%%%%%%%%%%%%%%%%%%%%%%%%%%%%%%%%%%%%%%%%%%%%%%%%%%%%%%%%%%
%%%%%%%%%%%%%%%%%%%%%%%%%%%%%%%%%%%%%%%%%%%%%%%%%%%%%%%%%%%%%%%%%%%%%%%%%%%%%%%%%%%%%%%%%%%%%%%%%%%%
%%%%%%%%%%%%%%%%%%%%%%%%%%%%%%%%%%%%%%%%%%%%%%%%%%%%%%%%%%%%%%%%%%%%%%%%%%%%%%%%%%%%%%%%%%%%%%%%%%%%
%%%%%%%%%%%%%%%%%%%%%%%%%%%%%%%%%%%%%%%%%%%%%%%%%%%%%%%%%%%%%%%%%%%%%%%%%%%%%%%%%%%%%%%%%%%%%%%%%%%%
%%%%%%%%%%%%%%%%%%%%%%%%%%%%%%%%%%%%%%%%%%%%%%%%%%%%%%%%%%%%%%%%%%%%%%%%%%%%%%%%%%%%%%%%%%%%%%%%%%%%
%%%%%%%%%%%%%%%%%%%%%%%%%%%%%%%%%%%%%%%%%%%%%%%%%%%%%%%%%%%%%%%%%%%%%%%%%%%%%%%%%%%%%%%%%%%%%%%%%%%%
%%%%%%%%%%%%%%%%%%%%%%%%%%%%%%%%%%%%%%%%%%%%%%%%%%%%%%%%%%%%%%%%%%%%%%%%%%%%%%%%%%%%%%%%%%%%%%%%%%%%
%%%%%%%%%%%%%%%%%%%%%%%%%%%%%%%%%%%%%%%%%%%%%%%%%%%%%%%%%%%%%%%%%%%%%%%%%%%%%%%%%%%%%%%%%%%%%%%%%%%%
%%%%%%%%%%%%%%%%%%%%%%%%%%%%%%%%%%%%%%%%%%%%%%%%%%%%%%%%%%%%%%%%%%%%%%%%%%%%%%%%%%%%%%%%%%%%%%%%%%%%
%%%%%%%%%%%%%%%%%%%%%%%%%%%%%%%%%%%%%%%%%%%%%%%%%%%%%%%%%%%%%%%%%%%%%%%%%%%%%%%%%%%%%%%%%%%%%%%%%%%%
\section{A finite generating set for $\M_4(N_{g,0})$}\label{4}

Denote $Y_{i,j}=Y_{\alpha_i,\alpha_{i,j}}$, $Y_{j,i}=Y_{\alpha_j,\alpha_{i,j}}$, $A_{i,j}=t_{\alpha_{i,j}}^4$, $B_{i,j}=t_{\beta_{i,j}}$, $C_{i,j;k}=t_{\alpha_{i,j}}^{-2}t_{Y_{k,i}(\alpha_{i,j})}^2$ or $t_{\alpha_{i,j}}^2t_{Y_{k,i}^{-1}(\alpha_{i,j})}^{-2}$ if $i<k<j$ or the others respectively, and $D_{i,j,k,l}=t_{\alpha_{i,j,k,l}}t_{Y_{j,i}Y_{k,l}^{-1}(\alpha_{i,j,k,l})}$ (see Figure~\ref{ABCD}).
Note that $A_{i,j}=(Y_{j,i}^{-1}Y_{i,j})^2=(Y_{j,i}Y_{i,j}^{-1})^2$, $B_{i,j}=Y_{i,j}^2=Y_{j,i}^2$ and $C_{i,j;k}=(Y_{k,j}Y_{k,i})^2$ or $(Y_{k,i}Y_{k,j})^2$ if $i<k<j$ or the others respectively (see~\cite{Sz1}).
Let $\Y$, $\A$, $\B$, $\C$ and $\D$ be sets as
\begin{eqnarray*}
\Y&=&\{Y_{i,j}\mid1\le{i}\le{g-1},1\le{j}\le{g}~\textrm{with}~i\neq{}j\},\\
\A&=&\{A_{i,j}\mid1\le{i<j}\le{}g\},\\
\B&=&\{B_{i,j}\mid1\le{i<j}\le{}g\},\\
\C&=&\{C_{i,j;k}\mid1\le{i<j}\le{}g,1\le{k}\le{g}~\textrm{with}~k\neq{}i,j\},\\
\D&=&\{D_{1,j,k,l}\mid1<j<k<l\le{}g\}.
\end{eqnarray*}
We take the total order of $\Y$ such that $Y_{i,j}<Y_{k,l}$ if $i<k$, or $i=k$ and $j<l$.
Let $\YB$ be
$$\YB=\{y_1y_2\cdots{}y_k\mid0\le{k}\le|\Y|,y_i\in\Y,y_1<y_2<\cdots<y_k\}.$$

%%%%%%%%%%%%%%%%%%%%%%%%%%%%%%%%%%%%%%%%%%%%%%%%%%%%%%%%%%%%%%%%%%%%%%%%%%%%%%%%%%%%%%%%%%%%%%%%%%%%
\begin{figure}[htbp]
\includegraphics{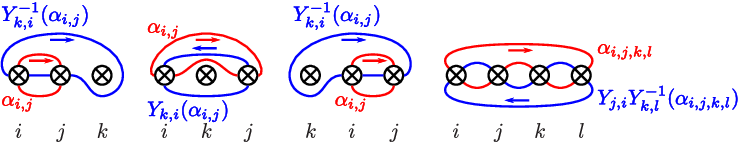}
\caption{Simple closed curves $Y_{k,i}^{\pm1}(\alpha_{i,j})$ and $Y_{j,i}Y_{k,l}^{-1}(\alpha_{i,j,k,l})$.}\label{ABCD}
\end{figure}
%%%%%%%%%%%%%%%%%%%%%%%%%%%%%%%%%%%%%%%%%%%%%%%%%%%%%%%%%%%%%%%%%%%%%%%%%%%%%%%%%%%%%%%%%%%%%%%%%%%%

In this section, we prove the following theorem.

%%%%%%%%%%%%%%%%%%%%%%%%%%%%%%%%%%%%%%%%%%%%%%%%%%%%%%%%%%%%%%%%%%%%%%%%%%%%%%%%%%%%%%%%%%%%%%%%%%%%
\begin{thm}\label{main-3}
For $g\geq4$, $\M_4(N_{g,0})$ is generated by $yxy^{-1}$ for $x\in\A\cup\B\cup\C\cup\D$ and $y\in\YB$.
\end{thm}
%%%%%%%%%%%%%%%%%%%%%%%%%%%%%%%%%%%%%%%%%%%%%%%%%%%%%%%%%%%%%%%%%%%%%%%%%%%%%%%%%%%%%%%%%%%%%%%%%%%%

First, we show the following lemma.

%%%%%%%%%%%%%%%%%%%%%%%%%%%%%%%%%%%%%%%%%%%%%%%%%%%%%%%%%%%%%%%%%%%%%%%%%%%%%%%%%%%%%%%%%%%%%%%%%%%%
\begin{lem}\label{M_2-gen}
For $g\geq4$, $\M_2(N_{g,0})$ is minimally generated by $\Y\cup\D$.
\end{lem}
%%%%%%%%%%%%%%%%%%%%%%%%%%%%%%%%%%%%%%%%%%%%%%%%%%%%%%%%%%%%%%%%%%%%%%%%%%%%%%%%%%%%%%%%%%%%%%%%%%%%

\begin{proof}
Hirose-Sato \cite{HS} proved that $\M_2(N_{g,0})$ is minimally generated by $\Y$ and $t_{\alpha_{1,j,k,l}}^2$ for $1<j<k<l\le{g}$.
Hence it suffices to show that $t_{\alpha_{1,j,k,l}}^2$ is in the group generated by $\Y\cup\D$.

Note that $t_{\alpha_{i_1,i_2}}^2=Y_{i_2,i_1}^{-1}Y_{i_1,i_2}$ (see \cite{Sz1}).
By a $3$-chain relation and braid relations, we calculate
\begin{eqnarray*}
t_{\alpha_{1,j,k,l}}t_{Y_{j,1}Y_{k,l}^{-1}(\alpha_{1,j,k,l})}^{-1}
&=&
(t_{\alpha_{1,j}}t_{\alpha_{j,k}}t_{\alpha_{k,l}})^4\\
&=&
t_{\alpha_{1,j}}^2t_{\alpha_{j,k}}t_{\alpha_{1,j}}^2t_{\alpha_{j,k}}t_{\alpha_{k,l}}t_{\alpha_{j,k}}t_{\alpha_{1,j}}^2t_{\alpha_{j,k}}t_{\alpha_{k,l}}\\
&=&
t_{\alpha_{1,j}}^2\cdot{}t_{\alpha_{j,k}}^2\cdot{}t_{\alpha_{j,k}}^{-1}t_{\alpha_{1,j}}^2t_{\alpha_{j,k}}\cdot{}t_{\alpha_{k,l}}^2\cdot{}t_{\alpha_{k,l}}^{-1}t_{\alpha_{j,k}}^2t_{\alpha_{k,l}}\cdot{}t_{\alpha_{k,l}}^{-1}t_{\alpha_{j,k}}^{-1}t_{\alpha_{1,j}}^2t_{\alpha_{j,k}}t_{\alpha_{k,l}}\\
&=&
Y_{j,1}^{-1}Y_{1,j}\cdot{}Y_{k,j}^{-1}Y_{j,k}\cdot{}Y_{j,1}(Y_{k,1}^{-1}Y_{1,k})Y_{j,1}^{-1}\\
&&
Y_{l,k}^{-1}Y_{k,l}\cdot{}Y_{k,j}(Y_{l,j}^{-1}Y_{j,l})Y_{k,j}^{-1}\cdot{}Y_{j,1}Y_{k,l}^{-1}(Y_{l,1}^{-1}Y_{1,l})Y_{k,l}Y_{j,1}^{-1}
\end{eqnarray*}
(see Fifure~\ref{3-chain}).
Note that $Y_{g,k}$ can be described as a product of elements in $\Y$ (see \cite{Sz2}).
Hence we have that $t_{\alpha_{1,j,k,l}}^2=t_{\alpha_{1,j,k,l}}t_{Y_{j,i}Y_{k,l}^{-1}(\alpha_{i,j,k,l})}^{-1}D_{1,j,k,l}$ is in the group generated by $\Y\cup\D$.

%%%%%%%%%%%%%%%%%%%%%%%%%%%%%%%%%%%%%%%%%%%%%%%%%%%%%%%%%%%%%%%%%%%%%%%%%%%%%%%%%%%%%%%%%%%%%%%%%%%%
\begin{figure}[htbp]
\includegraphics{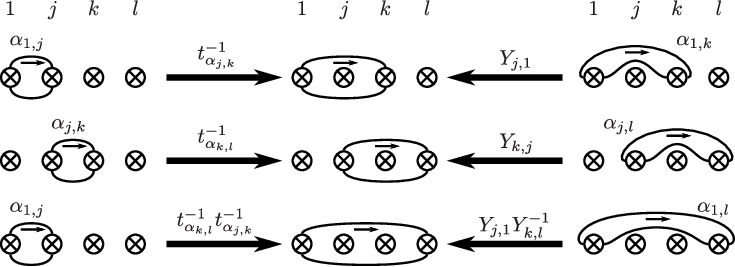}
\caption{$t_{\alpha_{j,k}}^{-1}(\alpha_{1,j})=Y_{j,1}(\alpha_{1,k})$, $t_{\alpha_{k,l}}^{-1}(\alpha_{j,k})=Y_{k,j}(\alpha_{j,l})$ and $t_{\alpha_{k,l}}^{-1}t_{\alpha_{j,k}}^{-1}(\alpha_{1,j})=Y_{j,1}Y_{k,l}^{-1}(\alpha_{1,l})$.}\label{3-chain}
\end{figure}
%%%%%%%%%%%%%%%%%%%%%%%%%%%%%%%%%%%%%%%%%%%%%%%%%%%%%%%%%%%%%%%%%%%%%%%%%%%%%%%%%%%%%%%%%%%%%%%%%%%%

Thus we finish the proof.
\end{proof}

Next, we show the following lemma.

%%%%%%%%%%%%%%%%%%%%%%%%%%%%%%%%%%%%%%%%%%%%%%%%%%%%%%%%%%%%%%%%%%%%%%%%%%%%%%%%%%%%%%%%%%%%%%%%%%%%
\begin{lem}\label{M_4-gen}
For $g\geq4$, $\M_4(N_{g,0})$ is generated by $y[x_1,x_2]y^{-1}$ and $yxy^{-1}$ for $x_1$, $x_2\in\Y$, $x\in\B\cup\D$ and $y=y_1y_2\cdots{}y_m\in\YB$ with $x_1<x_2$ and $y_1<y_2<\cdots<y_m<x_2$.
\end{lem}
%%%%%%%%%%%%%%%%%%%%%%%%%%%%%%%%%%%%%%%%%%%%%%%%%%%%%%%%%%%%%%%%%%%%%%%%%%%%%%%%%%%%%%%%%%%%%%%%%%%%

\begin{proof}
Let $G$ be the subgroup of $\M(N_{g,0})$ generated by the elements in Lemma~\ref{M_4-gen}.

We use the Reidemeister-Schreier method.
Imoto and the author \cite{IK2} proved that $[\GH_2(n),\GH_2(n)]=\G_4(n)$ for $n\ge3$.
By Theorem~\ref{main-1}, we have
$$\M_2(N_{g,0})/\M_4(N_{g,0})\cong\GH_2(g-1)/\G_4(g-1)=\GH_2(g-1)/[\GH_2(g-1),\GH_2(g-1)].$$
In addition, from a finite presentation for $\GH_2(g-1)$ given in \cite{HK}, we have that $\GH_2(g-1)/[\GH_2(g-1),\GH_2(g-1)]$ is isomorphic to $\left(\Z/2\Z\right)^{|\Y|}$.
We can regard $\YB=\left(\Z/2\Z\right)^{|\Y|}$ as a set.
Note that $\YB$ is a Schreier transversal for $\M_4(N_{g,0})$ in $\M_2(N_{g,0})$.
For $\varphi\in\M_2(N_{g,0})$, we denote by $\overline{\varphi}$ the element of $\YB$ corresponding to the image of $\varphi$ in $\left(\Z/2\Z\right)^{|\Y|}$.
By Lemma~\ref{M_2-gen}, $\M_4(N_{g,0})$ is generated by $\left\{yx^{\pm1}\overline{yx^{\pm1}}^{-1}\mid{}x\in\Y\cup\D,y\in\YB,\overline{yx^{\pm1}}\ne{}yx^{\pm1}\right\}$.

For $D_{1,j,k,l}\in\D$ and $y\in\YB$, we see
$$yD_{1,j,k,l}^{\pm1}\overline{yD_{1,j,k,l}^{\pm1}}^{-1}=yD_{1,j,k,l}^{\pm1}y^{-1}=(yD_{1,j,k,l}y^{-1})^{\pm1}\in{G}.$$
For $x\in\Y$ and $y\in\YB$, there are $z$, $w\in\Y$ such that either $y=zxw$, or $y=zw$ and $zxw\in\Y$.
When $y=zxw$, we see
\begin{eqnarray*}
yx\overline{yx}^{-1}&=&zxwxw^{-1}z^{-1}=z[x,w]z^{-1}\cdot{}zwx^2w^{-1}z^{-1},\\
yx^{-1}\overline{yx^{-1}}^{-1}&=&zxwx^{-1}w^{-1}z^{-1}=z[x,w]z^{-1}.
\end{eqnarray*}
When $y=zw$ and $zxw\in\Y$, we see
\begin{eqnarray*}
yx\overline{yx}^{-1}&=&zwxw^{-1}x^{-1}z^{-1}=z[w,x]z^{-1},\\
yx^{-1}\overline{yx^{-1}}^{-1}&=&zwx^{-1}w^{-1}x^{-1}z^{-1}=zwx^{-2}w^{-1}z^{-1}\cdot{}z[w,x]z^{-1}.
\end{eqnarray*}
Note that $z[w,x]z^{-1}=(z[x,w]z^{-1})^{-1}$ and $zwx^{-2}w^{-1}z^{-1}=(zwx^2w^{-1}x^{-1})^{-1}$.
Since $x^2$ is in $\B$, it suffices to show that $z[x,w]z^{-1}$ is in $G$.
Let $w=w_1w_2\cdots{}w_n$, where $w_1<w_2<\cdots<w_n\in\Y$.
Then we have
$$[x,w]=[x,w_1](w_1[x,w_2]w_1^{-1})\cdots(w_1w_2\cdots{}w_{n-1}[x,w_n]w_{n-1}^{-1}\cdots{}w_2^{-1}w_1^{-1})$$
(see \cite{IK1}).
Since we have that $zw_1w_2\cdots{}w_k$ is in $\YB$ and that $x<w_k<w_{k+1}$ for $1\le{k}\le{n-1}$, $z[x,w]z^{-1}$ is in $G$.

Thus we finish the proof.
\end{proof}

Next, we show the following lemma.

%%%%%%%%%%%%%%%%%%%%%%%%%%%%%%%%%%%%%%%%%%%%%%%%%%%%%%%%%%%%%%%%%%%%%%%%%%%%%%%%%%%%%%%%%%%%%%%%%%%%
\begin{lem}\label{[x_1,x_2]}
For any $x_1$ and $x_2\in\Y$ with $x_1<x_2$, $[x_1,x_2]$ is a product of $zx^{\pm1}z^{-1}$ for $x\in\A\cup\B\cup\C$ and $z=1$, $z_1$ or $z_2z_3\in\YB$ with $z_1\le{}x_2$ and $z_2<z_3<x_2$.
\end{lem}
%%%%%%%%%%%%%%%%%%%%%%%%%%%%%%%%%%%%%%%%%%%%%%%%%%%%%%%%%%%%%%%%%%%%%%%%%%%%%%%%%%%%%%%%%%%%%%%%%%%%

\begin{proof}
In this proof, we use relations on compositions of crosscap slides (for details, for instance see \cite{Sz1, O, KO2, IK1, KO3}).

Let $(x_1,x_2)=(Y_{i,j},Y_{j,i})$.
Then we calculate
\begin{eqnarray*}
[x_1,x_2]&=&Y_{i,j}^2(Y_{i,j}^{-1}Y_{j,i})^2Y_{j,i}^{-2}=B_{i,j}A_{i,j}^{-1}B_{i,j}^{-1}.
\end{eqnarray*}
Let $(x_1,x_2)=(Y_{i,j},Y_{i,k})$.
Then we calculate
\begin{eqnarray*}
[x_1,x_2]&=&(Y_{i,j}Y_{i,k})^2Y_{i,k}^{-2}\cdot{}Y_{i,k}Y_{i,j}^{-2}Y_{i,k}^{-1}\\
&=&
\left\{
\begin{array}{ll}
C_{j,k;i}B_{i,k}^{-1}\cdot{}x_2B_{i,j}^{-1}x_2^{-1}&(i<j<k),\\
x_1C_{j,k;i}x_1^{-1}\cdot{}B_{i,k}^{-1}\cdot{}x_2B_{j,i}^{-1}x_2^{-1}&(j<i<k),\\
C_{j,k;i}B_{k,i}^{-1}\cdot{}x_2B_{j,i}^{-1}x_2^{-1}&(j<k<i).
\end{array}
\right.
\end{eqnarray*}
Let $(x_1,x_2)=(Y_{i,j},Y_{k,j})$.
Then we calculate
\begin{eqnarray*}
[x_1,x_2]&=&Y_{i,j}(Y_{k,j}Y_{i,j}^{-1}Y_{k,j}^{-1})\\
&=&
\left\{
\begin{array}{ll}
Y_{i,j}(Y_{i,j}^{-2}Y_{i,k}^{-2}Y_{i,j})&(j<i<k),\\
Y_{i,j}(Y_{i,k}^{-2}Y_{i,j}^{-1})&(i<j<k),\\
Y_{i,j}(Y_{i,j}^{-2}Y_{i,k}^{-2}Y_{i,j})&(i<k<j)
\end{array}
\right.\\
&=&
\left\{
\begin{array}{ll}
Y_{i,j}^{-2}\cdot{}Y_{i,j}Y_{i,k}^{-2}Y_{i,j}^{-1}\cdot{}Y_{i,j}^2&(j<i<k),\\
Y_{i,j}Y_{i,k}^{-2}Y_{i,j}^{-1}&(i<j<k),\\
Y_{i,j}^{-2}\cdot{}Y_{i,j}Y_{i,k}^{-2}Y_{i,j}^{-1}\cdot{}Y_{i,j}^2&(i<k<j)
\end{array}
\right.\\
&=&
\left\{
\begin{array}{ll}
B_{j,i}^{-1}\cdot{}x_1B_{i,k}^{-1}x_1^{-1}\cdot{}B_{j,i}&(j<i<k),\\
x_1B_{i,k}^{-1}x_1^{-1}&(i<j<k),\\
B_{i,j}^{-1}\cdot{}x_1B_{i,k}^{-1}x_1^{-1}\cdot{}B_{i,j}&(i<k<j)
\end{array}
\right.
\end{eqnarray*}
(see Figure~\ref{y-comp}~(a)).
Let $(x_1,x_2)=(Y_{i,j},Y_{j,k})$.
Then we calculate
\begin{eqnarray*}
[x_1,x_2]&=&Y_{i,j}(Y_{j,k}Y_{i,j}^{-1}Y_{j,k}^{-1})\\
&=&
\left\{
\begin{array}{ll}
Y_{i,j}(Y_{i,k}^{-1}Y_{i,j}Y_{i,k})&(k<i<j),\\
Y_{i,j}(Y_{i,j}^{-2}Y_{i,k}^{-1}Y_{i,j}Y_{i,k}Y_{i,j}^2)&(i<k<j),\\
Y_{i,j}(Y_{i,k}^{-1}Y_{i,j}Y_{i,k})&(i<j<k)
\end{array}
\right.\\
&=&
\left\{
\begin{array}{ll}
Y_{i,j}Y_{i,k}^{-2}Y_{i,j}^{-1}(Y_{i,j}Y_{i,k})^2&(k<i<j),\\
Y_{i,j}^{-2}\cdot{}Y_{i,j}Y_{i,k}^{-2}Y_{i,j}^{-1}\cdot{}Y_{i,j}(Y_{i,k}Y_{i,j})^2Y_{i,j}^{-1}\cdot{}Y_{i,j}^2&(i<k<j),\\
Y_{i,j}Y_{i,k}^{-2}Y_{i,j}^{-1}(Y_{i,j}Y_{i,k})^2&(i<j<k)
\end{array}
\right.\\
&=&
\left\{
\begin{array}{ll}
x_1B_{k,i}^{-1}x_1^{-1}\cdot{}C_{k,j;i}&(k<i<j),\\
B_{i,j}^{-1}\cdot{}x_1B_{i,k}^{-1}x_1^{-1}\cdot{}x_1C_{k,j;i}x_1^{-1}\cdot{}B_{i,j}&(i<k<j),\\
x_1B_{i,k}^{-1}x_1^{-1}\cdot{}C_{j,k;i}&(i<j<k)
\end{array}
\right.
\end{eqnarray*}
(see Figure~\ref{y-comp}~(b)).
Let $(x_1,x_2)=(Y_{i,j},Y_{k,i})$.
Then we calculate
\begin{eqnarray*}
[x_1,x_2]&=&(Y_{i,j}Y_{k,i}Y_{i,j}^{-1})Y_{k,i}^{-1}\\
&=&
\left\{
\begin{array}{ll}
(Y_{k,i}^{-2}Y_{k,j}^{-1}Y_{k,i}^{-1}Y_{k,j}Y_{k,i}^2)Y_{k,i}^{-1}&(j<i<k),\\
(Y_{k,j}^{-1}Y_{k,i}^{-1}Y_{k,j})Y_{k,i}^{-1}&(i<j<k),\\
(Y_{k,i}^{-2}Y_{k,j}^{-1}Y_{k,i}^{-1}Y_{k,j}Y_{k,i}^2)Y_{k,i}^{-1}&(i<k<j)
\end{array}
\right.\\
&=&
\left\{
\begin{array}{ll}
Y_{k,i}^{-2}Y_{k,j}^{-2}\cdot{}Y_{k,j}Y_{k,i}^{-2}Y_{k,j}^{-1}(Y_{k,j}Y_{k,i})^2&(j<i<k),\\
(Y_{k,i}Y_{k,j})^{-2}Y_{k,i}Y_{k,j}^2Y_{k,i}^{-1}&(i<j<k),\\
Y_{k,i}^{-2}\cdot{}Y_{k,i}(Y_{k,j}Y_{k,i})^{-2}Y_{k,i}^{-1}\cdot{}Y_{k,i}Y_{k,j}^2Y_{k,i}^{-1}\cdot{}Y_{k,i}^2&(i<k<j)
\end{array}
\right.\\
&=&
\left\{
\begin{array}{ll}
B_{i,k}^{-1}B_{j,k}^{-1}\cdot{}Y_{k,j}B_{i,k}^{-1}Y_{k,j}^{-1}\cdot{}C_{j,i;k}&(j<i<k),\\
C_{i,j;k}^{-1}\cdot{}x_2B_{j,k}x_2^{-1}&(i<j<k),\\
B_{i,k}^{-1}\cdot{}x_2C_{i,j;k}^{-1}x_2^{-1}\cdot{}x_2B_{k,j}x_2^{-1}\cdot{}B_{i,k}&(i<k<j)
\end{array}
\right.
\end{eqnarray*}
(see Figure~\ref{y-comp}~(c)).
Note that $Y_{k,j}<x_2$ when $j<i<k$.
Let $(x_1,x_2)=(Y_{i,j},Y_{k,l})$.
Then we calculate
\begin{eqnarray*}
[x_1,x_2]&=&Y_{i,j}(Y_{k,l}Y_{i,j}^{-1}Y_{k,i}^{-1})\\
&=&
\left\{
\begin{array}{ll}
Y_{i,j}(Y_{i,l}^{-1}Y_{i,k}^2Y_{i,l}Y_{i,k}^2Y_{i,j}^{-1}Y_{i,k}^{-2}Y_{i,l}^{-1}Y_{i,k}^{-2}Y_{i,l})&(i<k<j<l),\\
Y_{i,j}(Y_{i,k}^{-2}Y_{i,l}^{-1}Y_{i,k}^{-2}Y_{i,l}Y_{i,j}^{-1}Y_{i,l}^{-1}Y_{i,k}^2Y_{i,l}Y_{i,k}^2)&(i<l<j<k),\\
Y_{i,j}(Y_{i,l}^{-1}Y_{i,k}^2Y_{i,l}Y_{i,k}^2Y_{i,j}^{-1}Y_{i,k}^{-2}Y_{i,l}^{-1}Y_{i,k}^{-2}Y_{i,l})&(j<l<i<k),\\
Y_{i,j}(Y_{i,l}^{-1}Y_{i,k}^2Y_{i,l}Y_{i,k}^2Y_{i,j}^{-1}Y_{i,k}^{-2}Y_{i,l}^{-1}Y_{i,k}^{-2}Y_{i,l})&(l<i<k<j)
\end{array}
\right.\\
&=&
\left\{
\begin{array}{ll}
Y_{i,j}Y_{i,l}^{-2}Y_{i,j}^{-1}\cdot{}Y_{i,j}Y_{i,l}Y_{i,k}^2Y_{i,l}^{-1}Y_{i,j}^{-1}\cdot{}Y_{i,j}Y_{i,l}^2Y_{i,j}^{-1}&\\
Y_{i,j}Y_{i,k}^2Y_{i,j}^{-1}\cdot{}Y_{i,k}^{-2}Y_{i,l}^{-2}\cdot{}Y_{i,l}Y_{i,k}^{-2}Y_{i,l}^{-1}\cdot{}Y_{i,l}^2&(i<k<j<l),\\
Y_{i,j}Y_{i,k}^{-2}Y_{i,j}^{-1}\cdot{}Y_{i,j}Y_{i,l}^{-2}Y_{i,j}^{-1}&\\
~[Y_{i,l}Y_{i,j}]^{-1}\cdot{}Y_{i,l}Y_{i,j}Y_{i,k}^{-2}Y_{i,j}^{-1}Y_{i,l}^{-1}\cdot[Y_{i,l}Y_{i,j}]&\\
Y_{i,j}Y_{i,l}^2Y_{i,j}^{-1}\cdot{}Y_{i,l}^{-2}\cdot{}Y_{i,l}Y_{i,k}^2Y_{i,l}^{-1}\cdot{}Y_{i,l}^2Y_{i,k}^2&(i<l<j<k),\\
Y_{i,j}Y_{i,l}^{-2}Y_{i,j}^{-1}\cdot{}Y_{i,j}Y_{i,l}Y_{i,k}^2Y_{i,l}^{-1}Y_{i,j}^{-1}\cdot{}Y_{i,j}Y_{i,l}^2Y_{i,j}^{-1}&\\
Y_{i,j}Y_{i,k}^2Y_{i,j}^{-1}\cdot{}Y_{i,k}^{-2}Y_{i,l}^{-2}\cdot{}Y_{i,l}Y_{i,k}^{-2}Y_{i,l}^{-1}\cdot{}Y_{i,l}^2&(j<l<i<k),\\
Y_{i,j}Y_{i,l}^{-2}Y_{i,j}^{-1}\cdot[Y_{i,l},Y_{i,j}]^{-1}\cdot{}Y_{i,l}Y_{i,j}Y_{i,k}^2Y_{i,j}^{-1}Y_{i,l}^{-1}&\\
~[Y_{i,l},Y_{i,j}]\cdot{}Y_{i,j}Y_{i,l}^2Y_{i,j}^{-1}\cdot{}Y_{i,j}Y_{i,k}^2Y_{i,j}^{-1}\\
Y_{i,k}^{-2}Y_{i,l}^{-2}\cdot{}Y_{i,l}Y_{i,k}^{-2}Y_{i,l}^{-1}\cdot{}Y_{i,l}^2&(l<i<k<j)
\end{array}
\right.\\
&=&
\left\{
\begin{array}{ll}
x_1B_{i,l}^{-1}x_1^{-1}\cdot{}x_1Y_{i,l}B_{i,k}Y_{i,l}^{-1}x_1^{-1}\cdot{}x_1B_{i,l}x_1^{-1}&\\
x_1B_{i,k}x_1^{-1}\cdot{}B_{i,k}^{-1}B_{i,l}^{-1}\cdot{}Y_{i,l}B_{i,k}^{-1}Y_{i,l}^{-1}\cdot{}B_{i,l}&(i<k<j<l),\\
x_1B_{i,k}^{-1}x_1^{-1}\cdot{}x_1B_{i,l}^{-1}x_1^{-1}&\\
~[Y_{i,l}Y_{i,j}]^{-1}\cdot{}Y_{i,l}x_1B_{i,k}^{-1}x_1^{-1}Y_{i,l}^{-1}\cdot[Y_{i,l}Y_{i,j}]&\\
x_1B_{i,l}x_1^{-1}\cdot{}B_{i,l}^{-1}\cdot{}Y_{i,l}B_{i,k}Y_{i,l}^{-1}\cdot{}B_{i,l}B_{i,k}&(i<l<j<k),\\
x_1B_{l,i}^{-1}x_1^{-1}\cdot{}x_1Y_{i,l}B_{i,k}Y_{i,l}^{-1}x_1^{-1}\cdot{}x_1B_{l,i}x_1^{-1}&\\
x_1B_{i,k}x_1^{-1}\cdot{}B_{i,k}^{-1}B_{l,i}^{-1}\cdot{}Y_{i,l}B_{i,k}^{-1}Y_{i,l}^{-1}\cdot{}B_{l,i}&(j<l<i<k),\\
x_1B_{l,i}^{-1}x_1^{-1}\cdot[Y_{i,l},Y_{i,j}]^{-1}\cdot{}Y_{i,l}x_1B_{i,k}x_1^{-1}Y_{i,l}^{-1}&\\
~[Y_{i,l},Y_{i,j}]\cdot{}x_1B_{l,i}x_1^{-1}\cdot{}x_1B_{i,k}x_1^{-1}\\
B_{i,k}^{-1}B_{l,i}^{-1}\cdot{}Y_{i,l}B_{i,k}^{-1}Y_{i,l}^{-1}\cdot{}B_{l,i}&(l<i<k<j)
\end{array}
\right.
\end{eqnarray*}
(see Figure~\ref{y-comp}~(d)).
Note that $x_1<Y_{i,l}<x_2$ when $i<k<j<l$ or $j<l<i<k$ and $Y_{i,l}<x_1$ when $i<l<j<k$ or $l<i<k<j$.
For the other cases we have $[x_1,x_2]=1$.

%%%%%%%%%%%%%%%%%%%%%%%%%%%%%%%%%%%%%%%%%%%%%%%%%%%%%%%%%%%%%%%%%%%%%%%%%%%%%%%%%%%%%%%%%%%%%%%%%%%%
\begin{figure}[htbp]
\subfigure[$Y_{k,j}Y_{i,j}^{-1}Y_{k,j}^{-1}$.]{\includegraphics{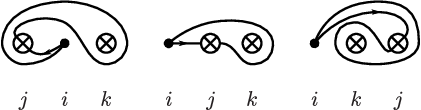}}\\
\subfigure[$Y_{j,k}Y_{i,j}^{-1}Y_{j,k}^{-1}$.]{\includegraphics{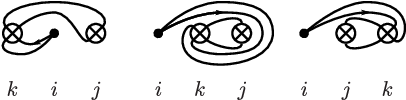}}\\
\subfigure[$Y_{i,j}Y_{k,i}Y_{i,j}^{-1}$.]{\includegraphics{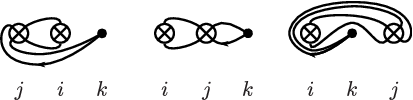}}\\
\subfigure[$Y_{k,l}Y_{i,j}^{-1}Y_{k,l}^{-1}$.]{\includegraphics{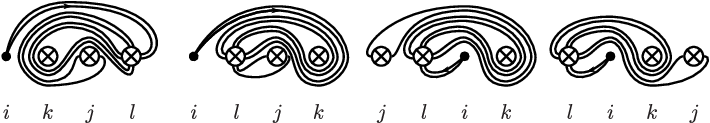}}
\caption{}\label{y-comp}
\end{figure}
%%%%%%%%%%%%%%%%%%%%%%%%%%%%%%%%%%%%%%%%%%%%%%%%%%%%%%%%%%%%%%%%%%%%%%%%%%%%%%%%%%%%%%%%%%%%%%%%%%%%

Thus we finish the proof.
\end{proof}

We now prove Theorem~\ref{main-3}.

Let $H$ be the subgroup of $\M(N_{g,0})$ generated by the elements in Theorem~\ref{main-3}.
By Lemmas~\ref{M_4-gen} and \ref{[x_1,x_2]}, it suffices to show that $yzxz^{-1}y^{-1}$ is in $H$ for $x\in\A\cup\B\cup\C$, $y=y_1y_2\cdots{}y_m\in\YB$ and $z=1$, $z_1$ or $z_2z_3\in\YB$ with $y_1<y_2<\cdots<y_m<x_2$, $z_1\le{}x_2$ and $z_2<z_3<x_2$.

Let $z=1$.
Since $yz$ is in $\YB$, we have that $yzxz^{-1}y^{-1}$ is in $H$.

Let $z=z_1$.
When $z_1>y_m$, since $yz$ is in $\YB$, we have that $yzxz^{-1}y^{-1}$ is in $H$.
When there is $1\le{k}\le{m}$ such that $y_{k-1}<z_1<y_k$, where $y_0=1$, we see
\begin{eqnarray}
yzxz^{-1}y^{-1}
&=&
y_1\cdots{}y_{k-1}[z_1,y_k\cdots{}y_m]^{-1}y_{k-1}^{-1}\cdots{}y_1^{-1}\label{1}\\
&&y_1\cdots{}y_{k-1}z_1y_k\cdots{}y_mxy_m^{-1}\cdots{}y_k^{-1}z_1^{-1}y_{k-1}^{-1}\cdots{}y_1^{-1}\label{2}\\
&&y_1\cdots{}y_{k-1}[z_1,y_k\cdots{}y_m]y_{k-1}^{-1}\cdots{}y_1^{-1}\label{3}.
\end{eqnarray}
Since $y_1\cdots{}y_{k-1}z_1y_k\cdots{}y_m$ is in $\YB$, we have that \eqref{2} is in $H$.
When there is $1\le{k}\le{m}$ such that $z_1=y_k$, we see
\begin{eqnarray}
yzxz^{-1}y^{-1}
&=&
y_1\cdots{}y_{k-1}[z_1,y_{k+1}\cdots{}y_m]y_{k-1}^{-1}\cdots{}y_1^{-1}\label{4}\\
&&y_1\cdots{}y_{k-1}y_{k+1}\cdots{}y_mz_1^2y_m^{-1}\cdots{}y_{k+1}^{-1}y_{k-1}^{-1}\cdots{}y_1^{-1}\label{5}\\
&&y_1\cdots{}y_{k-1}y_{k+1}\cdots{}y_mxy_m^{-1}\cdots{}y_{k+1}^{-1}y_{k-1}^{-1}\cdots{}y_1^{-1}\label{6}\\
&&y_1\cdots{}y_{k-1}y_{k+1}\cdots{}y_mz_1^{-2}y_m^{-1}\cdots{}y_{k+1}^{-1}y_{k-1}^{-1}\cdots{}y_1^{-1}\label{7}\\
&&y_1\cdots{}y_{k-1}[z_1,y_{k+1}\cdots{}y_m]^{-1}y_{k-1}^{-1}\cdots{}y_1^{-1}\label{8}
\end{eqnarray}
Since $y_1\cdots{}y_{k-1}y_k\cdots{}y_m$ is in $\YB$ and $z_1^2$ is in $\B$, we have that \eqref{5}, \eqref{6} and \eqref{7} are in $H$.

Let $z=z_2z_3$.
When $z_2>y_m$, since $yz$ is in $\YB$, we have that $yzxz^{-1}y^{-1}$ is in $H$.
When there is $1\le{k}\le{m}$ such that $y_{k-1}<z_2<y_k$, where $y_0=1$, we see
\begin{eqnarray}
yzxz^{-1}y^{-1}
&=&
y_1\cdots{}y_{k-1}[z_2,y_k\cdots{}y_m]^{-1}y_{k-1}^{-1}\cdots{}y_1^{-1}\label{9}\\
&&y_1\cdots{}y_{k-1}z_2y_k\cdots{}y_mz_3xz_3^{-1}y_m^{-1}\cdots{}y_k^{-1}z_2^{-1}y_{k-1}^{-1}\cdots{}y_1^{-1}\label{10}\\
&&y_1\cdots{}y_{k-1}[z_2,y_k\cdots{}y_m]y_{k-1}^{-1}\cdots{}y_1^{-1}.\label{11}
\end{eqnarray}
\eqref{10} is attributed to the case $z=z_1$.
When there is $1\le{k}\le{m}$ such that $z_2=y_k$, we see
\begin{eqnarray}
yzxz^{-1}y^{-1}
&=&
y_1\cdots{}y_{k-1}[z_2,y_{k+1}\cdots{}y_m]y_{k-1}^{-1}\cdots{}y_1^{-1}\label{12}\\
&&y_1\cdots{}y_{k-1}y_{k+1}\cdots{}y_mz_2^2y_m^{-1}\cdots{}y_{k+1}^{-1}y_{k-1}^{-1}\cdots{}y_1^{-1}\label{13}\\
&&y_1\cdots{}y_{k-1}y_{k+1}\cdots{}y_mz_3xz_3^{-1}y_m^{-1}\cdots{}y_{k+1}^{-1}y_{k-1}^{-1}\cdots{}y_1^{-1}\label{14}\\
&&y_1\cdots{}y_{k-1}y_{k+1}\cdots{}y_mz_2^{-2}y_m^{-1}\cdots{}y_{k+1}^{-1}y_{k-1}^{-1}\cdots{}y_1^{-1}\label{15}\\
&&y_1\cdots{}y_{k-1}[z_2,y_{k+1}\cdots{}y_m]^{-1}y_{k-1}^{-1}\cdots{}y_1^{-1}.\label{16}
\end{eqnarray}
Since $y_1\cdots{}y_{k-1}y_{k+1}\cdots{}y_m$ is in $\YB$ and $z_2^2$ is in $\B$, we have that \eqref{13} and \eqref{15} are in $H$.
In addition, \eqref{14} is attributed to the case $z=z_1$.

Finally, we consider \eqref{1}, \eqref{3}, \eqref{4}, \eqref{8}, \eqref{9}, \eqref{11}, \eqref{12} and \eqref{16}.
We see
\begin{eqnarray*}
\eqref{3}
&=&
y_1\cdots{}y_{k-1}[z_1,y_k]y_{k-1}^{-1}\cdots{}y_1^{-1}\\
&&y_1\cdots{}y_k[z_1,y_{k+1}]y_k^{-1}\cdots{}y_1^{-1}\\
&&\vdots\\
&&y_1\cdots{}y_{m-1}[z_1,y_m]y_{m-1}^{-1}\cdots{}y_1^{-1}.
\end{eqnarray*}
For $k\le{n}\le{m}$, by Lemma~\ref{[x_1,x_2]}, $[z_1,y_n]$ is a product of $z^\prime(x^\prime)^{\pm1}(z^\prime)^{-1}$ for $x^\prime\in\A\cup\B\cup\C$ and $z^\prime=1$, $z_1^\prime$ or $z_2^\prime{}z_3^\prime\in\YB$ with $z_1^\prime\le{}y_n$ and $z_2^\prime<z_3^\prime<y_n$.
By similar calculations, $y_1\cdots{}y_{n-1}z^\prime(x^\prime)^{\pm1}(z^\prime)^{-1}y_{n-1}^{-1}\cdots{}y_1^{-1}$, and so \eqref{3}, similarly also \eqref{1}, \eqref{4}, \eqref{8}, \eqref{9}, \eqref{11}, \eqref{12} and \eqref{16}, are attributed to \eqref{1}-\eqref{16}.
Remember that $y_m$ is less than $x_2$.
Now, $y_{n-1}$ is less than $y_m$.
Hence repeating similar calculations, \eqref{1}, \eqref{3}, \eqref{4}, \eqref{8}, \eqref{9}, \eqref{11}, \eqref{12} and \eqref{16} are attributed \eqref{2}, \eqref{5}, \eqref{6}, \eqref{7}, \eqref{13} and \eqref{15}. 
Therefore these are in $H$.

Thus, we conclude that $yzxz^{-1}y^{-1}$ is in $H$ for $x\in\A\cup\B\cup\C$, $y=y_1y_2\cdots{}y_m\in\YB$ and $z=1$, $z_1$ or $z_2z_3\in\YB$ with $y_1<y_2<\cdots<y_m<x_2$, $z_1<x_2$ and $z_2<z_3<x_2$, and so complete the proof.

\begin{rem}
By Theorem~\ref{main-1} and \cite{IK2}, for $l\ge3$, we have
$$\M_{2^{l-1}}(N_{g,0})/\M_{2^l}(N_{g,0})\cong\G_{2^{l-1}}(g-1)/\G_{2^l}(g-1)\cong(\Z/2\Z)^{(g-1)^2-1}$$
Let $\X$ be a finite generating set for $\M_{2^{l-1}}(N_{g,0})$ and let $\z$ be
$$\z=\{A_{i,g-1}^{2^{l-3}},C_{j,g;k}^{2^{l-3}}\mid1\le{i}\le{g-2},1\le{j,k}\le{g-1}~\textrm{with}~j\neq{}k\}.$$
We take one total order of $\z$.
Let $\ZB$ be
$$\ZB=\{z_1z_2\cdots{}z_k\mid0\le{k}\le|\z|,z_i\in\z,z_1<z_2<\cdots<z_k\}.$$
We can regard $\ZB=\left(\Z/2\Z\right)^{|\z|}=(\Z/2\Z)^{(g-1)^2-1}$ as a set.
Note that $\ZB$ is a Schreier transversal for $\M_{2^l}(N_{g,0})$ in $\M_{2^{l-1}}(N_{g,0})$.
For $\varphi\in\M_{2^{l-1}}(N_{g,0})$, we denote by $\overline{\varphi}$ the element of $\ZB$ corresponding to the image of $\varphi$ in $\left(\Z/2\Z\right)^{|\z|}$.
By the Reidemeister-Schreier method, inductively we can obtain the finite generating set for $\M_{2^l}(N_{g,0})$  as 
$$\left\{zx^{\pm1}\overline{zx^{\pm1}}^{-1}\mid{}x\in\X,z\in\ZB,\overline{zx^{\pm1}}\ne{}zx^{\pm1}\right\}.$$
\end{rem}

%%%%%%%%%%%%%%%%%%%%%%%%%%%%%%%%%%%%%%%%%%%%%%%%%%%%%%%%%%%%%%%%%%%%%%%%%%%%%%%%%%%%%%%%%%%%%%%%%%%%
%%%%%%%%%%%%%%%%%%%%%%%%%%%%%%%%%%%%%%%%%%%%%%%%%%%%%%%%%%%%%%%%%%%%%%%%%%%%%%%%%%%%%%%%%%%%%%%%%%%%
%%%%%%%%%%%%%%%%%%%%%%%%%%%%%%%%%%%%%%%%%%%%%%%%%%%%%%%%%%%%%%%%%%%%%%%%%%%%%%%%%%%%%%%%%%%%%%%%%%%%
%%%%%%%%%%%%%%%%%%%%%%%%%%%%%%%%%%%%%%%%%%%%%%%%%%%%%%%%%%%%%%%%%%%%%%%%%%%%%%%%%%%%%%%%%%%%%%%%%%%%
%%%%%%%%%%%%%%%%%%%%%%%%%%%%%%%%%%%%%%%%%%%%%%%%%%%%%%%%%%%%%%%%%%%%%%%%%%%%%%%%%%%%%%%%%%%%%%%%%%%%
%%%%%%%%%%%%%%%%%%%%%%%%%%%%%%%%%%%%%%%%%%%%%%%%%%%%%%%%%%%%%%%%%%%%%%%%%%%%%%%%%%%%%%%%%%%%%%%%%%%%
%%%%%%%%%%%%%%%%%%%%%%%%%%%%%%%%%%%%%%%%%%%%%%%%%%%%%%%%%%%%%%%%%%%%%%%%%%%%%%%%%%%%%%%%%%%%%%%%%%%%
%%%%%%%%%%%%%%%%%%%%%%%%%%%%%%%%%%%%%%%%%%%%%%%%%%%%%%%%%%%%%%%%%%%%%%%%%%%%%%%%%%%%%%%%%%%%%%%%%%%%
%%%%%%%%%%%%%%%%%%%%%%%%%%%%%%%%%%%%%%%%%%%%%%%%%%%%%%%%%%%%%%%%%%%%%%%%%%%%%%%%%%%%%%%%%%%%%%%%%%%%
%%%%%%%%%%%%%%%%%%%%%%%%%%%%%%%%%%%%%%%%%%%%%%%%%%%%%%%%%%%%%%%%%%%%%%%%%%%%%%%%%%%%%%%%%%%%%%%%%%%%
\section{On a finite generating set for $\M_d(N_{g,n})$ with $n\ge1$}\label{5}

For any $d\ge2$, $g\ge1$ and $n\ge0$, since $\M(N_{g,n})$ can be finitely generated and $\M_d(N_{g,n})$ is an finite index subgroup of $\M(N_{g,n})$, $\M_d(N_{g,n})$ can be also finitely generated.
We put a finite generating set $\e_0$ for $\M_d(N_{g,0})$.
For $n\ge1$, let $\e_n$ be the set consisting of any lifts by $\M_d(N_{g,n})\to\M_d(N_{g,0})$ of any elements of $\e_0$, and for $1\le{l}\le{n}$, let
\begin{eqnarray*}
\f_l&=&\left\{t_{\alpha_{i,g}}^d,t_{\alpha_{i;l}}^d,t_{\delta_l},t_{\epsilon_{j,l}},t_{\zeta_{k,l}},t_{\overline{\zeta}_{k,l}},t_{\eta_{i_1,i_2}}
\left|
\begin{array}{l}
1\le{i}\le{g-1},1\le{j}\le{g},\\
1\le{i_1<i_2}\le{g-1},1\le{k}<l
\end{array}
\right.
\right\},\\
\g_l&=&\{(t_{\alpha_{1,g}}^{-d}t_{\alpha_{1;l}}^d)^{m_1}(t_{\alpha_{2,g}}^{-d}t_{\alpha_{2;l}}^d)^{m_2}\cdots(t_{\alpha_{g-1,g}}^{-d}t_{\alpha_{g-1;l}}^d)^{m_{g-1}}\mid0\le{m_i}\le{d-1}\},\\
\h_n&=&\bigcup_{1\le{l}\le{n}}\{zyz^{-1}\mid{}y\in\f_l,z\in\g_l\},
\end{eqnarray*}
where twist simple closed curves are as shown in Figures~\ref{alpha}, \ref{beta-gamma-delta-epsilon-zeta} and \ref{ez}.
We define $\f_0=\g_0=\h_0=\emptyset$.
In addition, denote by $\h_m\subset\M_d(N_{g,n})$ lift by $\M_d(N_{g,n})\to\M_d(N_{g,m})$ of $\h_m\subset\M_d(N_{g,m})$ for $1\le{m}<n$.

In this section, we prove the following theorem.

\begin{thm}\label{gen-n>0}
For $d\ge2$, $g\ge1$ and $n\ge1$, $\M_d(N_{g,n})$ is generated by $\e_n$ and $\h_n$.
\end{thm}

Remember the homomorphism $\theta:\pi_1^+(N_{g,n-1})\to(\Z/d\Z)^g$ defined in Subsection~\ref{3.2}.
Let $x_{i_1,i_2,\dots,i_k}=x_{i_1}x_{i_2}\cdots{}x_{i_k}$ and let
$$\GG=\{x_{1,g}^{m_1}x_{2,g}^{m_2}\cdots{}x_{g-1,g}^{m_{g-1}}\mid0\le{m_i}\le{d-1}\}.$$
First, we prove the following proposition.

\begin{prop}\label{gen-ker-theta}
For $d\ge2$, $g\ge1$ and $n\ge1$, $\ker\theta$ is generated by $wx_{i,g}^dw^{-1}$, $wx_{j,j}w^{-1}$, $wy_kw^{-1}$, $wz_kw^{-1}$ and $wx_{i_1,i_2,g}^2w^{-1}$ for $1\le{i}\le{g-1}$, $1\le{j}\le{g}$, $1\le{k}\le{n-1}$, $1\le{i_1<i_2}<g$ and $w\in\GG$.
\end{prop}

\begin{proof}
Remember that $\pi_1^+(N_{g,n-1})$ is generated by
$$X=\{x_{i,g},x_{j,j},y_k,z_k\mid1\le{i}\le{g-1},1\le{j}\le{g},1\le{k}\le{n-1}\},$$
and that $\textrm{Im}\theta$ is isomorphic to $\displaystyle\bigoplus_{1\le{i}\le{g-1}}\Z/d\Z[\theta(x_{i,g})]$.
Then we notice that $\GG$ is a Schreier transversal for $\ker\theta$ in $\pi_1^+(N_{g,n-1})$.
Hence by the Reidemeister-Schreier method, $\ker\theta$ is generated by
$$\{wx^{\pm1}\overline{wx^{\pm1}}^{-1}\mid{}x\in{X},w\in\GG,\overline{wx^{\pm1}}\neq{}wx^{\pm1}\},$$
where $\overline{wx^{\pm1}}$ is the element of $\GG$ corresponding to the image of $wx^{\pm1}\in\pi_1^+(N_{g,n-1})$ in $\displaystyle\bigoplus_{1\le{i}\le{g-1}}\Z/d\Z[\theta(x_{i,g})]$.
Let $I$ be the subgroup of $\pi_1^+(N_{g,n-1})$ generated by the elements in Proposition~\ref{gen-ker-theta}.
We show that for any $x\in{X}$ and $w\in\GG$, $wx^{\pm1}\overline{wx^{\pm1}}^{-1}$ is in $I$.

For $x=x_{j,j}$, $y_k$ and $z_k$, we see
$$wx^{\pm1}\overline{wx^{\pm1}}^{-1}=wx^{\pm1}w^{-1}\in{}I.$$
For $x=x_{i,g}$ and $w=x_{1,g}^{m_1}x_{2,g}^{m_2}\cdots{}x_{g-1,g}^{m_{g-1}}\in\GG$, let
\begin{eqnarray*}
w_1&=&x_{1,g}^{m_1}x_{2,g}^{m_2}\cdots{}x_{i-1,g}^{m_{i-1}},\\
w_2&=&x_{i+1,g}^{m_{i+1}}x_{i+2,g}^{m_{i+2}}\cdots{}x_{g-1,g}^{m_{g-1}}.
\end{eqnarray*}
When $m_i=0$, we see
\begin{eqnarray*}
wx^{-1}\overline{wx^{-1}}^{-1}
&=&w_1w_2x^{-1}w_2^{-1}x^{1-d}w_1^{-1}\\
&=&w_1[w_2,x^{-1}]w_1^{-1}\cdot{}w_1x^{-d}w_1^{-1}.
\end{eqnarray*}
When $m_i=d-1$, we see
\begin{eqnarray*}
wx\overline{wx}^{-1}
&=&w_1x^{d-1}w_2xw_2^{-1}w_1^{-1}\\
&=&w_1x^dw_1^{-1}\cdot{}w_1[x^{-1},w_2]w_1^{-1}.
\end{eqnarray*}
For the other cases, we see
\begin{eqnarray*}
wx^{\pm1}\overline{wx^{\pm1}}^{-1}
&=&w_1x^{m_i}w_2x^{\pm1}w_2^{-1}x^{-m_i\mp1}w_1^{-1}\\
&=&w_1x^{m_i}[w_2,x^{\pm1}]x^{-m_i}w_1^{-1}.
\end{eqnarray*}
By the argument similar to the proof of Lemma~\ref{M_4-gen}, it follows that $[x^{\pm1},w_2]$ is the product of
$$X_{j,m}=x_{i+1,g}^{m_{i+1}}x_{i+2,g}^{m_{i+2}}\cdots{}x_{j-1,g}^{m_{j-1}}x_{j,g}^{m}[x^{\pm1},x_{j,g}]x_{j,g}^{-m}x_{j-1,g}^{-m_{j-1}}\cdots{}x_{i+2,g}^{-m_{i+2}}x_{i+1,g}^{-m_{i+1}}$$
for $i<j<g$ and $0\le{m}<m_j$.
More explicitly,
$$[x^{\pm1},w_2]=(X_{i+1,0}X_{i+1,1}\cdots{}X_{i+1,m_{i+1}-1})\cdots(X_{g-1,0}X_{g-1,1}\cdots{}X_{g-1,m_{g-1}-1}).$$
Hence it suffices to show that for $i<j<g$ and $0\le{m}<m_j$, $w_1x^lX_{j,m}x^{-l}w_1^{-1}$ is in $I$, where $l=m_i$ or $0$.

Let $y=x_{i+1,g}^{m_{i+1}}x_{i+2,g}^{m_{i+2}}\cdots{}x_{j-1,g}^{m_{j-1}}x_{j,g}^{m}$.
We calculate
\begin{eqnarray*}
~[x,x_{j,g}]&=&x_ix_gx_jx_gx_g^{-1}x_i^{-1}x_g^{-1}x_j^{-1}\\
&=&x_ix_gx_jx_jx_gx_i(x_i^{-1}x_g^{-1}x_j^{-1})^2\\
&=&x_ix_gx_jx_jx_g^{-1}x_i^{-1}\cdot{}x_ix_gx_gx_gx_g^{-1}x_i^{-1}\cdot{}x_ix_i\cdot{}x_jx_g(x_ix_jx_g)^{-2}x_g^{-1}x_j^{-1}\\
&=&xx_{j,j}x^{-1}\cdot{}xx_{g,g}x^{-1}\cdot{}x_{i,i}\cdot{}x_{j,g}x_{i,j,g}^{-2}x_{j,g}^{-1},\\
~[x^{-1},x_{j,g}]&=&(x^{-1}[x,x_{j,g}]x)^{-1}\\
&=&x^{-1}x_{j,g}x_{i,j,g}^2x_{j,g}^{-1}x\cdot{}x^{-1}x_{i,i}^{-1}x\cdot{}x_{g,g}^{-1}\cdot{}x_{j,j}^{-1}.
\end{eqnarray*}
Hence $X_{j,m}$ is a product of $yzy^{-1}$ for $z=x_{i^\prime,i^\prime}$, $x^{\pm1}x_{i^\prime,i^\prime}x^{\mp1}$, $x_{j,g}x_{i,j,g}^2x_{j,g}^{-1}$ and $x^{-1}x_{j,g}x_{i,j,g}^2x_{j,g}^{-1}x$.
For $z=x_{i^\prime,i^\prime}$ and $x_{j,g}x_{i,j,g}^2x_{j,g}^{-1}$, since $w_1x^ly$ and $w_1x^lyx_{j,g}$ are in $\GG$, we have that $w_1x^lyzy^{-1}x^{-l}w_1^{-1}$ is in $I$.
In addition, we see
\begin{eqnarray*}
w_1x^lyx^{\pm1}
&=&w_1x^l[y,x^{\pm1}]x^{-l}w_1^{-1}\cdot{}w_1x^{l\pm1}y.
\end{eqnarray*}
Note that
\begin{eqnarray*}
w_1x^{-1}y&=&w_1x^{-d}w_1^{-1}\cdot{}w_1x^{d-1}y\equiv{}w_1x^{d-1}y\pmod{I},\\
w_1x^dy&=&w_1x^dw_1^{-1}\cdot{}w_1y\equiv{}w_1y\pmod{I}.
\end{eqnarray*}
Moreover, similarly it follows that $[y,x^{\pm1}]=[x^{\pm1},y]^{-1}$ is the product of $X_{j^\prime,m^\prime}$ for $i<j^\prime<j$ (resp. $j^\prime=j$) and $0\le{m^\prime}<m_{j^\prime}$ (resp. $0\le{}m^\prime<m$).
Hence for $z=x^{\pm1}x_{i^\prime,i^\prime}x^{\mp1}$ and $x^{-1}x_{j,g}x_{i,j,g}^2x_{j,g}^{-1}x$, since $w_1x^{l\pm1}y$ and $w_1x^{l-1}yx_{j,g}$ are in $\GG$ modulo $I$, we have that $w_1x^lyzy^{-1}x^{-l}w_1^{-1}$ is a product of elements of $I$ and $w_1x^lX_{j^\prime,m^\prime}x^{-l}w_1^{-1}$.
We notice $j^\prime<j$ or $m^\prime<m$.
By the similar calculation, it follows that $w_1x^lX_{j^\prime,m^\prime}x^{-l}w_1^{-1}$ is a product of elements of $I$ and $w_1x^lX_{j^{\prime\prime},m^{\prime\prime}}x^{-l}w_1^{-1}$, where $j^{\prime\prime}<j^\prime$ or $m^{\prime\prime}<m^\prime$.
Repeating, we conclude that $w_1x^lX_{j,m}x^{-l}w_1^{-1}$ is a product of elements of $I$ and $w_1x^lX_{i+1,0}x^{-l}w_1^{-1}$.
Furthermore, since $X_{i+1,0}=[x^{\pm1},x_{i+1,g}]$ is a product of $x_{i^\prime,i^\prime}$, $x^{\pm1}x_{i^\prime,i^\prime}x^{\mp1}$, $x_{i+1,g}x_{i,i+1,g}^2x_{j,g}^{-1}$ and $x^{-1}x_{i+1,g}x_{i,j,g}^2x_{i+1,g}^{-1}x$, we have that $w_1x^lX_{i+1,0}x^{-l}w_1^{-1}$ is in $I$, and so $w_1x^lX_{j,m}x^{-l}w_1^{-1}$ is in $I$.

Therefore for any $x\in{X}$ and $w\in\GG$, $wx^{\pm1}\overline{wx^{\pm1}}^{-1}$ is in $I$.
Thus we finish the proof.
\end{proof}

We now prove Theorem~\ref{gen-n>0}.

For $n\ge1$, suppose that $\M_d(N_{g,n-1})$ is generated by $\e_{n-1}$ and $\h_{n-1}$.
Note that this assumption is true if $n=1$ since $\h_0=\emptyset$.
$\M_d(N_{g,n})$ is generated by $t_{\delta_n}$ and lifts by $\CC$ of the generators of $\CC(\M_d(N_{g,n}))$.
Moreover, by Proposition~\ref{PFC}, $\CC(\M_d(N_{g,n}))$ is generated by images by $\PP$ of all generators of $\ker\theta$ and lifts by $\FF$ of the generators of $\M_d(N_{g,n-1})$.
Hence $\M_d(N_{g,n})$ is generated by
\begin{enumerate}
\item	lift by $\CC\circ\FF$ of $\e_{n-1}$,
\item	lift by $\CC\circ\FF$ of $\h_{n-1}$,
\item	lifts by $\CC$ of $\PP(wx_{i,g}^dw^{-1})$, $\PP(wx_{j,j}w^{-1})$, $\PP(wy_kw^{-1})$, $\PP(wz_kw^{-1})$, $\PP(wx_{i_1,i_2,g}^2w^{-1})$ for $1\le{i}\le{g-1}$, $1\le{j}\le{g}$, $1\le{k}\le{n-1}$, $1\le{i_1<i_2}<g$ and $w\in\GG$,
\item	$t_{\delta_n}$.
\end{enumerate}
(1) is $\e_n$ and (2) is $\h_{n-1}$.
In addition, since $\CC(z(t_{\alpha_{i,g}}^{-d}t_{\alpha_{i;n}}^d)z^{-1})=\PP(wx_{i,g}^dw^{-1})$, $\CC(zt_{\epsilon_{j,n}}z^{-1})=\PP(wx_{j,j}w^{-1})$, $\CC(zt_{\zeta_{k,n}}z^{-1})=\PP(wy_kw^{-1})$, $\CC(zt_{\overline{\zeta}_{k,n}}z^{-1})=\PP(wz_kw^{-1})$ and $\CC(zt_{\eta_{i_1,i_2}}z^{-1})=\PP(wx_{i_1,i_2,g}^2w^{-1})$ for some $z\in\g_n$, we have that (3) is in $\h_n\setminus(\h_{n-1}\cup\{t_{\delta_n}\})$.
Therefore $\M_d(N_{g,n})$ is generated by $\e_n$ and $\h_n$.

Thus we complete the proof.

\begin{rem}
Finite generating sets for $\M_2(N_{g,0})$ and $\T_2(N_{g,0})$ are given (see \cite{Sz2,HS,KO1}).
In addition, A finite generating set for $\M_4(N_{g,0})$ ($\M_{2^l}(N_{g,0})$) is given in this paper.
Therefore, we obtain the finite generating sets for $\M_2(N_{g,n})$, $\T_2(N_{g,n})$ and  $\M_4(N_{g,n})$ ($\M_{2^l}(N_{g,n})$), for $g\ge4$ and $n\ge1$.
\end{rem}

%%%%%%%%%%%%%%%%%%%%%%%%%%%%%%%%%%%%%%%%%%%%%%%%%%%%%%%%%%%%%%%%%%%%%%%%%%%%%%%%%%%%%%%%%%%%%%%%%%%%
%%%%%%%%%%%%%%%%%%%%%%%%%%%%%%%%%%%%%%%%%%%%%%%%%%%%%%%%%%%%%%%%%%%%%%%%%%%%%%%%%%%%%%%%%%%%%%%%%%%%
%%%%%%%%%%%%%%%%%%%%%%%%%%%%%%%%%%%%%%%%%%%%%%%%%%%%%%%%%%%%%%%%%%%%%%%%%%%%%%%%%%%%%%%%%%%%%%%%%%%%
%%%%%%%%%%%%%%%%%%%%%%%%%%%%%%%%%%%%%%%%%%%%%%%%%%%%%%%%%%%%%%%%%%%%%%%%%%%%%%%%%%%%%%%%%%%%%%%%%%%%
%%%%%%%%%%%%%%%%%%%%%%%%%%%%%%%%%%%%%%%%%%%%%%%%%%%%%%%%%%%%%%%%%%%%%%%%%%%%%%%%%%%%%%%%%%%%%%%%%%%%
%%%%%%%%%%%%%%%%%%%%%%%%%%%%%%%%%%%%%%%%%%%%%%%%%%%%%%%%%%%%%%%%%%%%%%%%%%%%%%%%%%%%%%%%%%%%%%%%%%%%
%%%%%%%%%%%%%%%%%%%%%%%%%%%%%%%%%%%%%%%%%%%%%%%%%%%%%%%%%%%%%%%%%%%%%%%%%%%%%%%%%%%%%%%%%%%%%%%%%%%%
%%%%%%%%%%%%%%%%%%%%%%%%%%%%%%%%%%%%%%%%%%%%%%%%%%%%%%%%%%%%%%%%%%%%%%%%%%%%%%%%%%%%%%%%%%%%%%%%%%%%
%%%%%%%%%%%%%%%%%%%%%%%%%%%%%%%%%%%%%%%%%%%%%%%%%%%%%%%%%%%%%%%%%%%%%%%%%%%%%%%%%%%%%%%%%%%%%%%%%%%%
%%%%%%%%%%%%%%%%%%%%%%%%%%%%%%%%%%%%%%%%%%%%%%%%%%%%%%%%%%%%%%%%%%%%%%%%%%%%%%%%%%%%%%%%%%%%%%%%%%%%
\section*{Acknowledgements}

The author would like to express his gratitude to Susumu Hirose and Genki Omori for their valuable comments and meaningful discussions.

%%%%%%%%%%%%%%%%%%%%%%%%%%%%%%%%%%%%%%%%%%%%%%%%%%%%%%%%%%%%%%%%%%%%%%%%%%%%%%%%%%%%%%%%%%%%%%%%%%%%
%%%%%%%%%%%%%%%%%%%%%%%%%%%%%%%%%%%%%%%%%%%%%%%%%%%%%%%%%%%%%%%%%%%%%%%%%%%%%%%%%%%%%%%%%%%%%%%%%%%%
%%%%%%%%%%%%%%%%%%%%%%%%%%%%%%%%%%%%%%%%%%%%%%%%%%%%%%%%%%%%%%%%%%%%%%%%%%%%%%%%%%%%%%%%%%%%%%%%%%%%
%%%%%%%%%%%%%%%%%%%%%%%%%%%%%%%%%%%%%%%%%%%%%%%%%%%%%%%%%%%%%%%%%%%%%%%%%%%%%%%%%%%%%%%%%%%%%%%%%%%%
%%%%%%%%%%%%%%%%%%%%%%%%%%%%%%%%%%%%%%%%%%%%%%%%%%%%%%%%%%%%%%%%%%%%%%%%%%%%%%%%%%%%%%%%%%%%%%%%%%%%

\end{document}